\documentclass[11pt]{amsart}
\usepackage{amsmath,amsfonts,amsthm}
\usepackage{amssymb,latexsym}
\usepackage[colorlinks=true, linkcolor=blue,
urlcolor=blue]{hyperref}
\usepackage{graphicx}
\usepackage[all]{xy}
\usepackage{graphics}
\usepackage{amsthm}
\usepackage{tikz-cd}

\usepackage{layout}
\theoremstyle{plain}
\newtheorem{theorem}{Theorem}[section]
\newtheorem{lemma}[theorem]{Lemma}
\newtheorem{proposition}[theorem]{Proposition}
\newtheorem{corollary}[theorem]{Corollary}

\theoremstyle{definition}

\newtheorem{definition}[theorem]{Definition}
\newtheorem{example}[theorem]{Example}

\newtheorem{remark}[theorem]{Remark}

\theoremstyle{remark}

\numberwithin{equation}{section}
\newcounter{TmpEnumi}

\title{On Topological  Rank of Factors of Cantor Minimal Systems}

\author{Nasser Golestani and Maryam Hosseini}

\address{Department of Pure Mathematics, Faculty of Mathematical Sciences, 
  Tarbiat Modares University,
 Tehran\\ Iran}
 \address{School of Mathematics, Institute for Research in Fundamental Sciences (IPM), P. O. Box 19395-5746, Tehran, Iran}
 \email{n.golestani@modares.ac.ir}
\address{School of Mathematics, Institute for Research in Fundamental Sciences (IPM), P. O. Box 19395-5746, Tehran, Iran}
\email{maryhoseini@ipm.ir}

\begin{document}
\begin{abstract} A Cantor minimal system is of finite topological rank if  it has a Bratteli-Vershik representation whose number of vertices per level is uniformly bounded. We  prove that if the topological rank of a minimal dynamical system on a Cantor set is finite then all its minimal Cantor factors  have finite topological rank as well. This gives an affirmative answer to a question posed by Donoso,  Durand, Maass, and Petite.
\end{abstract}
\maketitle
\section{Introduction}
%
%
%
A {\it Cantor minimal system} is a pair $(X, T)$ where $T$ is a homeomorphism on a Cantor space $X$. Inspired by the definition of  {\it Rokhlin towers} for approximating an {\it ergodic system} by a finite union of towers of {\it measurable sets} \cite{O}, the definition of {\it Kakutani-Rokhlin towers} was created for Cantor minimal systems  \cite{HPS, putnam}.
 So as a topological analogue, a Kakutani-Rokhlin partition for a Cantor minimal system  is a finite union of towers of {\it clopen sets} in which every level of a  tower is mapped onto its upper level up to the top level. The difference in this analogy is that a Kakutani-Rokhlin partition topologically  covers the space $X$ and the dynamical system is generated point-wisely  by a nested sequence of Kakutani-Rokhlin towers that  their bases converge to a point. When the number of towers of each partition in the sequence is uniformly bounded, the associated system is called of {\it finite topological rank,} denoted by ${\rm rank}_{\rm top}(X, T)<\infty$. 

 \smallskip
 Some examples of finite topological rank Cantor minimal systems  are symbolic systems generated by the coding of  interval exchange transformations, Substitution subshifts and  linearly recurrent systems on Cantor sets, \cite{D, DHS, D5, GJ}. It is a folklore that {\it odometers} (minimal isometries  on Cantor sets) are the only  rank one  minimal Cantor systems. 

 \smallskip

 A system of finite topological rank has zero entropy and bounded number of invariant ergodic measures \cite{BKMS}. Moreover, rank of the additive group of {\it continuous spectrum} (that is, the eigenvalues of the Koopman operator $U_T(f)=f\circ T$ defined on $C(X)$) as an abelian group,   is dominated by the topological rank of the system \cite{CDP, GHH}. 
 
 The Kakutani-Rokhlin partitions associated to a (essentially) minimal  Cantor system  $(X, T)$ ``approximates"  that homeomorphism by  a {\it shift} map (up to the top levels of the towers). However, as the system   is eventually  generated   by a  sequence of partitions, it can be far from being a {\it subshift} \cite{GJ1}. 
 A remarkable theorem of T. Downarowicz and A. Maass    \cite{DM} states  that a Cantor minimal system of finite topological rank  is an odometer or a  subshift. The latter holds if the rank is bigger than one.

\smallskip
 A dynamical system on a Cantor set is  subshift if and only if it is expansive \cite{Kurka}. Expansivity  is not  inherited by  the  {\it topological factor} of an expansive system.   For instance, odometers are topological factors of any system with non-trivial  continuous rational spectrum. Even though one can assure that a system is not an isometry it is not sufficient to say that it is expansive (to see an example look at \cite{GJ}). 
 But by the main result of \cite{DM}, this can be guaranteed for a finite rank subshifts if one can show that its  topological factors  have finite rank as well. This is a corollary of  the main result  of this note. We recall that a dynamical system is essentially minimal if it has a unique minimal subsystem.
\vspace{-0.09cm}

\begin{theorem}\label{main}
Let $(X,\,T)$ be an essentially minimal Cantor system of finite topological rank and $(Y,\,S)$ be a minimal system on  a Cantor set such that for some continuous  map $\alpha:X\rightarrow Y$, $\alpha\circ T=S\circ \alpha$. Then 
$${\rm rank}_{\rm top}(Y,\,S)<\infty.$$
\end{theorem}
Indeed, we prove that ${\rm rank}_{\rm top}(Y, S)\leq 3\, {\rm rank}_{\rm top}(X, T)$.
For the special case that $(X, T)$ is minimal  and $(Y, S)$  is subshift, there has been recently a proof to show that ${\rm rank}_{\rm top}(Y, S)<\infty$ in \cite{bastian}. However,  we don't have any condition on $(Y, S)$ here except that it is a Cantor factor of $(X, T)$. Moreover, minimality of $(X, T)$ is not necessary in our proof.

Theorem \ref{main}  answers  Question 8.4 posed in \cite{DMPD} affirmatively. 
Moreover, one can get the dichotomy  of  the main result of \cite{DM} for 
the Cantor factors of finite topological rank Vershik systems.

\begin{corollary}\label{main1}
Every minimal Cantor factor of a finite topological rank essentially minimal 
  system on a compact totally disconnected metrizable space  is an odometer or a subshift.
\end{corollary}

 Recently, $S$-adic subshifts, that are Cantor systems generated by a sequence of morphisms between finite alphabets, have also been studied by  people who are interested in ``approximating" finite rank minimal subshifts with primitive substitutions \cite{DL, DMPD}. In \cite[Theorem 4.1]{DMPD} the authors proved that every minimal $S$-adic subshift with {\it bounded alphabet rank } is a topological factor of a finite topological rank minimal Cantor system. Combining this result with Theorem \ref{main} one can say that
 \begin{corollary}
 Any minimal $S$-adic subshift with bounded alphabet rank is conjugate to a subshift of finite topological rank.
  \end{corollary}

\smallskip

%
To prove Theorem \ref{main}, we use Bratteli-Vershik representations of $(X, T)$ and $(Y, S)$. Then we apply the representation of the factor map $\alpha$ in terms of an
 {\it  ordered premorphism} between the associated ordered Bratteli diagrams. This notion has been recently defined in \cite{AEG}. 
  Here we will show how the existence of an ordered premorphism from an ordered Bratteli diagram $B_1$  to $B_2$  may reduce  the  number of vertices of  levels of $B_1$ needed to construct an ordered Bratteli diagram equivalent to $B_1$ whose rank is dominated by $3\,{\rm rank}(B_2)$.  
 \medskip
 
The paper naturally starts by recalling basic definitions and tools in Section 2. In Section 3 the notion of an ordered premorphism is defined via a sequence of morphisms and the main result of \cite{AEG} that is the main tool of this paper is presented. Then in Section 4, Theorem \ref{main}  is proved. In Section 5, a combinatorial condition is given for verification of conjugacy

%
%
%

\section{Preliminaries}

\subsection{Topological Dynamical Systems on Cantor Set.}
A {\it topological dynamical system} is a pair $(X, T)$ where $X$ is a compact metric space and $T$ is an onto homeomorphism of $X$. 
If $Z$ is an invariant closed subset of $X$,
then $(Z,T)$ is called a {\it subsystem.}
The orbit of a  point $x\in X$, denoted by ${\mathcal O}(x)$, is the set $\{T^nx\}_{n\in\mathbb{Z}}$. If $X$ is a Cantor space (that is, a nonempty compact metrizable totally disconnected space with no isolated points)
then the system is called a {\it Cantor system}. Two topological dynamical systems $(X, T)$ and $(Y, S)$ are  {\it semi-conjugate} if there exists a continuous map $\alpha: X\rightarrow Y$ such that $\alpha\circ T=S\circ \alpha$. In this case $(Y, S)$ is called a {\it factor} of $(X, T)$, and $(X, T)$ is called an extension of $(Y, S)$.
\medskip

When $X=A^\mathbb Z$ where $A$ is a finite alphabet of cardinality $n\geq 2$ and $X$ is equipped with the compact  product topology, that makes $X$ homeomorphic to the Cantor set, together with
the shift map $T$  acting on the bi-infinite sequences of $X$, then $(X, T)$ is called a {\it shift} system. 
Every subsystem of a shift is called a {\it subshift} system.

For a topological dynamical system $(X, T)$ if the orbits of all points  are dense
in $X$, then the system is called {\it minimal}. This is equivalent to the absence of
non-trivial invariant closed subsets. When  $(X, T)$  has a unique minimal subsystem,
 the system is called {\it essentially minimal}. Every essentially minimal system on a Cantor set has  realizations by sequences $\{\mathcal T_n\}_{n\geq 1}$ of {\it Kakutani-Rokhlin} (briefly called {\it K-R}) {\it partitions} \cite{HPS}. Each K-R partition $ \mathcal T_n=\cup_{i=1}^k\cup_{j=1}^hB_{ij}$, is a finite union of towers, $\cup_{j=1}^hB_{ij}$, of clopen sets, $B_{ij}$, so that $T(B_{ij})=B_{ij+1}$ for $j<h$. The {\it base} of the tower is $\cup_{i=1}^kB_{i1}$. The towers construction is based on the first return time of the points of bases to them. 
 When the number of towers in $\mathcal T_n$ of a  sequence $\{\mathcal T_n\}_{n\geq 1}$ is uniformly bounded,
 that is, ${\rm max}_n \#{\mathcal T}_n=d<\infty$, the system is said to be of {\it finite 
 topological rank},
 and the minimum such $d$ is called the rank of $(X,T)$ \cite{DMPD}. 
We use the notation
  ${\rm rank}_{\rm top}(X, T)=d$.
\medskip

\medskip

\subsection{Ordered Bratteli Diagrams.}
A {\it Bratteli diagram} is an infinite directed graph $B=((V_i)_{i\geq 0},\, (E_i)_{i\geq 1})$ where $V=\dot{\bigcup}_{i\geq 0} V_i$ is the set of vertices with $V_0=\{v_0\}$ and for each $i\geq 1$, $E_i$ is the edges between $V_{i-1}$ and $V_{i}$.  
Each $V_{i}$ and each $E_i$ is a finite 
nonempty set.
There are two maps $r, s: E\rightarrow V$, called the range and the source maps respectively, with $r(E_i)\subset V_i$ and $s(E_i)\subset V_{i-1}$. A vertex $v\in V_i$ is connected to a vertex $w\in V_{i-1}$ if there exists an edge $e\in E_i$ such that $r(e)=v$ and $s(e)=w$. In this way, for each $n\geq 1$ there is a $|V_n|\times |V_{n-1}|$ incidence matrix $A_n$ whose entry $a_{ij}$ counts the number of edges between $v_i\in V_n$ and $w_j\in V_{n-1}$. 
We assume that every row and every column of each $A_n$ 
is nonzero.

For $m,n\geq 0$ with
$m<n$, let $E_{m,n}$ be the set of
finite paths from $V_{m}$ to $V_{n}$,
that is, $E_{m,n}$ consists of the tuples
$(e_{m+1},\ldots,e_{n})$
 where $ e_{i}\in E_{i}$ for $ i=m+1,\ldots,n$, and $r(e_{i})=s(e_{i+1})$
 for $i=m+1,\ldots,n-1$.
In particular,  $E_{m,m}=\{(v,v)\mid v\in V_{m}\}$  is an edge set from $V_{m}$ to itself.

For a strictly increasing sequence of integers $n=\{n_k\}_{k\geq 0}$
with $n_0=0$, one can define the {\it  telescoping} of the diagram along $n$ by defining a new Bratteli diagram $B'=((V'_i)_{i\geq 0},\, (E'_i)_{i\geq 1})$ in which for every $i\geq 1$, $V'_i=V_{n_i}$, $E'_i=E_{n_i, n_{i+1}}$,
and $V'_0=V_0$. So the incidence matrices of $B'$ are $A'_i= A_{n_i}\times A_{n_{i}-1}\times\cdots\times A_{n_{i-1}+1}$. A Bratteli diagram is called {\it  simple} if there exists a telescoping of that along a sequence such that all the incidence matrices have just positive entries. 
\medskip

An {\it ordered Bratteli diagram,} $B=((V_i)_{i\geq 0},\, (E_i)_{i\geq 1}, \geq)$, is a Bratteli
diagram
$((V_i)_{i\geq 0},\, (E_i)_{i\geq 1})$ together
 with a partial ordering on the set of its edges in which two edges
$e$ and $e'$ are comparable if and only if $r(e)=r(e')$. In fact, for every $n\geq 1$
 and every $v\in V_n\setminus V_0$, 
$r^{-1}(v)$ is linearly ordered.  For each $v$, the edge
with the largest (smallest) number in the ordering of $r^{-1}(v)$ is called the {\it max
edge\/} ({\it min edge\/}). For every telescoping of $B$ there is an induced ordering on the edge set. In fact, 
 $(e_{k+1}, e_{k+2},\dots,e_l)>(f_{k+1}, f_{k+2},\dots,f_l)$ as two finite paths in $E_{k,\ell}$
if $r(e_\ell)=r(f_\ell)$ and  there exists some $i$ with $k+1\leq i\leq l$ such that for all
$i<j\leq l$, $e_j=f_j$ and  $e_i>f_i$.

Let $B=((V_i)_{i\geq 0},\, (E_i)_{i\geq 1}, \geq)$ be an ordered  Bratteli diagram. The set of {\it infinite paths} is 

$$
X_B=\{(e_1, e_2,\dots):\ \ e_i\in E_i,\ r(e_i)=s(e_{i+1}),\ \
i=1,2,\dots\}.
$$
Two paths are {\it cofinal\/} if  all but finitely many of  their edges agree.  The set $X_B$ is equipped with  the usual compact product topology that its basis consists of  cylinder sets  of the from
$$
U(e_1, e_2,\dots,e_k)=\{(f_1, f_2, \dots)\in X_B:\  f_i=e_i,\ 1\leq
i\leq k\}.
$$
 $X_B$ is a compact Hausdorff space with a
countable basis consisting of clopen sets and is homeomorphic to the Cantor set if it is infinite
and $B$ is simple. Let $X_B^{ \max}$ denote the set of
all those elements $(e_1, e_2, \dots)$  in $X_B$ such that each  $e_n$ is
a max edge, and define $X_B^{\min}$ analogously.  An ordered Bratteli
diagram is  called {\it properly ordered\/} if it is simple  and $X_B^{\max}$ and $X_B^{\min}$ each contains only one
element; when this occurs, the max and min paths are  denoted
$x_{\max}$ and $x_{\min}$ respectively. For any Bratteli diagram, there
exists an ordering which makes it   properly ordered  \cite{HPS}.

\medskip

Let $B=((V_i)_{i\geq 0},\, (E_i)_{i\geq 1}, \geq)$ be a properly ordered Bratteli diagram. The {\it Vershik\/
{\rm(or} adic\/{\rm)} map} is the homeomorphism $\varphi_B:
X_B\rightarrow X_B$ wherein    $\varphi_B(x_{\max})=x_{\min}$, and for any other
point $(e_1, e_2,\dots)\neq x_{\max}$,  the map sends the path to its successor \cite{HPS}; in particular,  let $k$ be the smallest number that
$e_k$ is not a max edge and let $f_k$ be the immediate successor of $e_k$. Then 
$\varphi_B(e_1, e_2,\dots)=(f_1, \dots, f_{k-1},f_k,e_{k+1},e_{k+2},\dots)$,
where $(f_1,\dots, f_{k-1})$ is the min path in $E_{0, k-1}$ having the range $s(f_k)$.

Using Kakutani-Rokhlin partitions for  Cantor minimal systems, Herman, Putnam and Skau 
proved that:

\begin{theorem}[\cite{HPS}]\label{HPS}
Let $(X,\,T)$ be a Cantor minimal system. Then $T$ is
topologically conjugate to a Vershik map on a Bratteli compactum $X_B$  of
a simple properly ordered Bratteli diagram $B$. Furthermore, given
$x_0\in X$ we may choose the conjugating map $f:X\rightarrow X_B$ so that
$f(x_0)$ is the unique infinite max path in $B$.
\end{theorem}


\subsection{S-adic Representation of Minimal Subshifts.} 
One of the models to represent a minimal system of finite topological rank is the {\it $S$-adic} representation \cite{DL, DMPD}. We use the definitions of \cite{DMPD}. 

Let $\{A_n\}_{n\geq 0}$   be a sequence of finite alphabets and suppose that $\tau=(\tau_n:A_{n+1}\rightarrow A_n^*)_{n\geq 0}$ is a {\it directive sequence} of {\it morphisms} such  that for every $a\in A_{n+1}$, $\tau_n(a)$ is not the empty word. Then  there is a sequence of $|A_{n+1}|\times |A_n|$ matrices $M_n$ (also denoted by $M_{\tau_n}$) so that for every $n\geq 0$ each entry  $(M_n)_{ij}$ counts the number of occurrences of the $j$th letter of $A_n$ in $\tau_n(a_i)$, $a_i\in A_{n+1}$. When all the matrices are positive the sequence of morphisms is called {\it positive}. $\tau$ is {\it proper} if every $\tau_n$ is proper and the latter means that for every $n$, there exists letters $a, b$  in $A_n$ such that for all $c\in A_{n+1}$, $\tau_n(c)$ starts with $a$ and ends up with $b$. Moreover, $\tau$ is called {\it primitive} if for every $n\geq 1$ there exists some $N\geq n$ such that $M_{\tau_{[n,N)}}>0$ where $\tau_{[n,N)}=\tau_n\circ \tau_{n+1}\circ\cdots\circ\tau_{N-1}$. For every $n\geq 0$ let
$$\mathcal L^{(n)}(\tau)=\big\{w\in A_n^*: \ \ \exists \, N>n \ \, \exists\, a\in A_N \ \ w \ {\rm occurs \ in\ } \tau_{[n,N)}(a)\big\}.$$
Suppose that $X_\tau^{(n)}$ is the set of points $x\in A_n^{\mathbb Z}$ so that all the factors of $x$ belong to $\mathcal L^{(n)}(\tau)$. Then $(X_\tau^{(n)}, \sigma)$ is a subshift and if $\tau$ is primitive 
it will be a minimal susbshift.  The minimal subshift  $(X, \sigma):=(X_\tau^{(0)}, \sigma)$ is called the {\it $S$-adic subshift } generated by the directive sequence $\tau$.
\medskip

Let $B=((V_k)_{k\geq 0},\, (E_k)_{k\geq 1},\geq)$ be an  ordered Bratteli diagram.   
There exists a sequence of morphisms $\sigma=(\sigma_i^B: V_i\rightarrow V_{i-1}^*)_{i\geq 1}$ defined by, for $i\geq 2$, 
\[
\sigma_i^B(v)=s(e_1(v))s(e_2(v))\cdots s(e_k(v)),
\]
 where $\{e_j(v)\colon \ j=1, \ldots, k(v)\}$ is the ordered set of edges
in $E_{i}$ with range $v$ and for $i=1$, $\sigma^B_1: V_1^*\rightarrow E_1^*$, $\sigma_1^B(v)=e_1(v)\cdots e_\ell(v)$ where $e_1(v), \dots, e_\ell(v)$ are all the edges with range $v$.

 Each $\sigma_i^B$
 extends to $V_{i}^{*}$ by concatenation. For every $i, j\in\mathbb{N}$ with $i<j$, we define $\sigma_{[i,j]}:V_j^*\rightarrow V_i^*$ by $\sigma_{[i,j]}=\sigma_{i+1}\circ\sigma_{i+2}\circ\cdots\circ\sigma_{j}$.
Also, let $\sigma_{[i,i]}:V_i^*\rightarrow V_i^*$ be the identity map.

 \smallskip
\begin{proposition}[\cite{DMPD}, Proposition 4.6]\label{sadic}
Let $(X,T)$ be a minimal Cantor system given by a Bratteli-Vershik representation $B$. If $(X,T)$ is subshift, then, after an appropriate telescoping, the $S$-adic subshift generated by the sequence of morphisms
$\sigma^B =(\sigma_n^B : V_{i}\rightarrow V_{i-1}^*)_{i\geq 1}$ read on $B$ is  conjugate to $(X, T)$.
\end{proposition}

\subsection{Fine-Wilf Theorem}
An elementary fact that we 
need in the sequel is a  form of the well-known Fine-Wilf theorem \cite{FW} that 
follows easily from an induction argument and the original form of that  theorem,
and we state it here for the sake of completeness. 
Let us first recall that if  $A$ is an alphabet and if
$w=w_{1} \cdots w_{n}$ with  $w_{1},\ldots, w_{n}\in A$, then
$w$  is called \emph{periodic} with period $p\leq n$ whenever  
 $w_i=w_{i+p}$,    for every $1\leq i\leq n-p$.

\begin{lemma}\label{lem_FW}
Let $A$ be a finite alphabet and
let $k\in\mathbb{N}$. If $w\in A^*$ has periods $p_1, p_2, \dots, p_k$ 
such that $|w|\geq p_1+ p_2+\cdots+ p_k -{\rm gcd}(p_1, p_2, \dots, p_k)$, then $w$ is periodic 
with period ${\rm gcd}(p_1, p_2, \dots, p_k)$.
\end{lemma}

\section{Factoring and Bratteli Diagrams}
Let $(X, T)$ and $(Y, S)$ be two Cantor minimal systems such that for some continuous map $\alpha: X\rightarrow Y$, $\alpha\circ T=S\circ \alpha$. It is natural to realize this relation between the two systems in terms of a relation between their associated  Bratteli diagrams. For  the special case that $\alpha$ is almost one-to-one (that
is, $\alpha$ is one-to-one for a generic point in $X$) some  characterization  have been proved in  \cite{sugisaki}. For the general case, 
 in \cite{AEG} the authors defined the notion of
 an {\it ordered premorphism} between two ordered Bratteli diagrams  and proved the equivalence of the existence of an ordered premorphism between two properly ordered Bratteli diagrams with the existence of a factor map  between their associated Bratteli-Vershik systems. Here for the sake of completeness and to be more precise we first recall the definition of an ordered premorphism and  Proposition 4.6 of \cite{AEG}  and  then, to be prepared for the proof of Theorem \ref{main},  we will show how such an ordered premorphism  induces a sequence of morphisms between the two diagrams. There will be an example at the end of the section. 
  
 \begin{definition}[\cite {AEG}, Definition 3.1]\label{defpreobd}
Let $B_1=(V,E,\geq)$ and $B_2=(W, S,\geq')$
be  ordered Bratteli diagrams.
By an \emph{ordered} \emph{premorphism}
(or just a \emph{premorphism}
if there is no confusion) $f: B_1\to B_2$ we mean
 a triple $(F, (f_{n})_{n=0}^{\infty},\geq )$ where
$(f_{n})_{n=0}^{\infty}$ is a cofinal
(i.e., unbounded)
 sequence of  positive integers with
$f_{0}=0\leq f_{1}\leq f_{2}\leq\cdots$,
 $F$ consists of a disjoint union
 $F_{0}\cup F_{1}\cup F_{2}\cup\cdots$ together with a pair of range and source maps
 $r:F\to W$, $s:F\to V$, and
  $\geq$ is a partial ordering on $F$ such that:

\begin{enumerate}
\item\label{defprebd_it0}
 each $F_{n}$ is a non-empty finite set, $s(F_{n})\subseteq V_{n}$,
 $r(F_{n})\subseteq W_{f_{n}}$,  $F_{0}$
 is a singleton,
$s^{-1}\{v\}$ is non-empty for all $v$ in $V$, and
$r^{-1}\{w\}$ is non-empty for all $w$ in $W$;
\item\label{defpreobd_it1}
$e,e'\in F$ are comparable if and only if $r(e)=r(e')$, and $\geq$ is a linear order
on
$r^{-1}\{w\}$, for
all $w\in W$;

\item\label{defpreobd_it2}
 the diagram of $f:B_1\to B_2$,

  \[
\xymatrix{V_{0}\ar[r]^{E_{1}}\ar[d]_{F_{0}}
 &V_{1}\ar[r]^-{E_{2}}\ar[d]_{F_{1}} &V_{2}\ar[r]^-{E_{3}}\ar[d]_{F_{2}} &\cdots \ \ \  \\
 W_{f_{0}}\ar[r]_{S_{f_{0},f_{1}}}
 &W_{f_{1}}\ar[r]_{S_{f_{1},f_{2}}} &W_{f_{2}}\ar[r]_{S_{f_{2},f_{3}}}&\cdots \   ,
 }
\]
commutes.
The ordered commutativity of the diagram of $f$  means that
for each
 $n\geq 0$,
 $E_{n+1}\circ F_{n+1}\cong F_{n}\circ S_{f_{n},f_{n+1}}$, i.e.,
 there is a (necessarily unique) bijective map
from $E_{n+1}\circ F_{n+1}$ to
 $F_{n}\circ S_{f_{n},f_{n+1}}$
 preserving the order and intertwining the respective
source and range
maps.
\end{enumerate}
\end{definition}
\medskip
To see how  the ordered premorphism $f:B_1\to B_2$ induces a factoring $\alpha:X_{B_2}\to X_{B_1}$ between the two Vershik systems, 
let $x=(s_{1},s_{2},\ldots)$ be  an infinite path in $X_{B_2}$.
Define the path $\alpha(x)=(e_{1},e_{2},\ldots)$ in $X_{B_1}$ as follows.
Fix $n\geq 1$.
By Definition~\ref{defpreobd}, the  diagram

 \[
\xymatrix{V_{0}\ar[r]^{E_{0,n}}\ar[d]_{F_{0}}
 &V_{n}\ar[d]^{F_{n}} \\
 W_{0}\ar[r]_{S_{0,f_{n}}}
 &W_{f_{n}}
 }
\]
commutes,  that is, $F_{0}\circ S_{0,f_{n}}\cong E_{0,n}\circ F_{n}$. Thus, there is a unique path
$(e_{1},e_{2},\ldots,e_{n},d_{n})$ in $E_{0,n}\circ F_{n}$ (in fact $(e_1, e_2, \ldots, e_n)\in E_{0,n}$ and $d_n\in F_n$), corresponding to
the path $(s_{0},s_{1},\ldots,s_{f_{n}})$ in
$F_{0}\circ S_{0,f_{n}}$
where $s_{0}$ is the unique element of $F_{0}$. So
the path $\alpha(x)=(e_{1},e_{2},\ldots)$ in $X_{B}$
is associated to the path $x=(s_{1},s_{2},\ldots)$  in $X_{B_2}$. 

The 
correspondence of factor maps between  two Vershik systems and ordered premorphisms between the associated properly ordered Bratteli diagrams was established in the proof of the following proposition.

 \begin{proposition}[\cite {AEG}, Proposition 4.6]\label{AEG_prop}
Let $(X,T)$ and $(Y, S)$ be Cantor minimal systems, and let $x\in X$ and $y\in Y$ . Suppose that $B_1$ and $B_2$ are Bratteli-Vershik models for $(Y, S, y)$ and $(X, T, x)$ respectively. The following statements are equivalent:\begin{enumerate}
\item there is a factor map $\alpha: (X,T)\rightarrow (Y,S)$ with $\alpha(x) = y$;
\item there is an (ordered) premorphism $f$ from $B_1$ to $B_2$.
\end{enumerate}
More precisely, there is a one-to-one correspondence between the set of factor maps $\alpha$ as in (1)  and the set of equivalence classes of ordered premorphisms $f$ from $B_1$ to $B_2$.
\end{proposition}

  \begin{definition}\label{new}
 Let $B_1=(V,\, E,\geq)$ and $B_2=(W,\, S, \,\geq')$ be two ordered Bratteli diagrams with an ordered premorphism $f:B_1\rightarrow B_2$ between them.  One can describe an induced  sequence of morphisms
\[
\tau=(\tau_n:W_{n}\rightarrow V_n^*)_{n\geq 0}
\]
  between the two diagrams as follows.
Let $f=(F, (f_{n})_{n=0}^{\infty},\geq)$.
Assume that $f_{n}=n$ for every $n\geq 0$.
(We can always make this assumption 
using a telescoping of $B_2$ along a 
strictly increasing subsequence of
$(f_{n})_{n=0}^{\infty}$). 
Let $V=\cup_{n=0}^{\infty}V_{n}$ and 
$W=\cup_{n=0}^{\infty}W_{n}$
be the sets of vertices of $B_{1}$
and $B_{2}$, respectively.
Then for every $n\geq 0$, the morphism $\tau_n$ is  a map $\tau_n: W_n\rightarrow V_n^{*}$ that for every vertex $w\in W_n$, $\tau(w)$ represents the ordered set of vertices of $V_{n}$ connected $w$ via $f$.
More precisely, let 
$F=\cup_{n=0}^{\infty}F_{n}$  be the decomposition of the set of edges
of $f$ and let $\{g_{1},g_{2},\ldots,g_{m}\}$
be the ordered set of edges in $F_n$
with range $w$. Then 
\[
\tau_n(w)=s(g_1)s(g_2)\cdots s(g_m).
\]

We extend $\tau_n$ to 
$W_{n}^{*}$ by concatenation.
Observe that the essential property
of the ordered premorphism $f$ which is the ordered commutativity, reads as
\begin{equation}\label{equ_ordcom_mor}
\tau_n\circ \sigma_{n+1}^{B_2}=\sigma_{n+1}^{B_1}\circ \tau_{n+1},
\end{equation}
for all $n\geq 1$.
Conversely, every sequence of morphisms
satisfying \eqref{equ_ordcom_mor}
induces an ordered premorphism between the two diagrams,

however,  this direction is not needed in this
note.

 \end{definition}

\begin{example}

Figure \ref{1} shows an example  of ordered premorphism between two ordered Bratteli diagrams. The left diagram is associated to the Sturmain system with  rotattion number $\theta={{1+\sqrt{5}}\over 2}$ and the right diagram is a substitution which is orbit equivalent to the Sturmian system with rotation number $\theta/2$. 

\begin{figure}
\begin{center}
\begin{tikzpicture}[scale=1.1]
\hspace{-1cm}\path (3,13.5) node[inner sep=0pt](x'10) {} (10,13.5) node[inner sep=0pt](x10) {}

(4,9.5) node[inner sep=0pt](z'11) {}  (2,9.5) node[inner sep=0pt](z'12) {} 
(12,9.5) node[inner sep=0pt](z11) {}  (10,9.5) node[inner sep=0pt](z12) {} 
(8,9.5) node[inner sep=0pt](z13){}

(4,5.5) node[inner sep=0pt](w'11) {} (2,5.5) node[inner sep=0pt](w'12) {}
(12,5.5) node[inner sep=0pt](w11) {} (10,5.5) node[inner sep=0pt](w12) {}
(8,5.5) node[inner sep=0pt](w13) {}

(2,1.5) node[inner sep=0pt](y'12) {} (4,1.5) node[inner sep=0pt](y'11) {}
(10,1.5) node[inner sep=0pt](y12) {} (12,1.5) node[inner sep=0pt](y11) {}
 (8,1.5) node[inner sep=0pt](y13) {}

(12,1) node[inner sep=0pt](p11) {} (12,0.5) node[inner sep=0pt](p12){} (12,0) node[inner sep=0pt](p13){}
(10,1) node[inner sep=0pt](p21) {} (10,0.5) node[inner sep=0pt](p22){} (10,0) node[inner sep=0pt](p23){}
(8,1) node[inner sep=0pt](p31) {} (8,0.5) node[inner sep=0pt](p32){} (8,0) node[inner sep=0pt](p33){}

(4,1) node[inner sep=0pt](p'11) {} (4,0.5) node[inner sep=0pt](p'12){} (4,0) node[inner sep=0pt](p'13) {}
(2,1) node[inner sep=0pt](p'21) {} (2,0.5) node[inner sep=0pt](p'22){} (2,0) node[inner sep=0pt](p'23) {}
;

{\color{red}\filldraw (x10) circle [radius=0.1];
\filldraw (z11) circle [radius=0.1];
\filldraw (z12)  circle[radius=0.1];
\filldraw (z13) circle [radius=0.1];
\filldraw (w11) circle [radius=0.1];
\filldraw (w12) circle [radius=0.1];
\filldraw (w13)  circle[radius=0.1];
\filldraw (y13)  circle[radius=0.1];
\filldraw (y11)  circle[radius=0.1];
\filldraw (y12)  circle[radius=0.1];
\filldraw (p11)  circle[radius=0.03];
\filldraw (p12)  circle[radius=0.03];
\filldraw (p13)  circle[radius=0.03];
\filldraw (p21)  circle[radius=0.03];
\filldraw (p22)  circle[radius=0.03];
\filldraw (p23)  circle[radius=0.03];
\filldraw (p31)  circle[radius=0.03];
\filldraw (p32)  circle[radius=0.03];
\filldraw (p33)  circle[radius=0.03];}
\node at (8.6,9.55) {$x_1$};
\node at (11.2,9.55) {$y_1$};
\node at (13.2,9.55) {$z_1$};
\node at (8.6,5.55) {$x_2$};
\node at (11.2,5.55) {$y_2$};
\node at (13.2,5.55) {$z_2$};
\node at (8.6,1.55) {$x_3$};
\node at (11.2,1.55) {$y_3$};
\node at (13.2,1.55) {$z_3$};

{\color{blue}\filldraw (x'10) circle [radius=0.1];
\filldraw (z'11) circle [radius=0.1];
\filldraw (z'12)  circle[radius=0.1];
\filldraw (w'11) circle [radius=0.1];
\filldraw (w'12) circle [radius=0.1];
\filldraw (y'11)  circle[radius=0.1];
\filldraw (y'12)  circle[radius=0.1];
\filldraw (p'11)  circle[radius=0.03];
\filldraw (p'12)  circle[radius=0.03];
\filldraw (p'13)  circle[radius=0.03];
\filldraw (p'21)  circle[radius=0.03];
\filldraw (p'22)  circle[radius=0.03];
\filldraw (p'23)  circle[radius=0.03];
}
\node at (4.6,9.55) {$v_1$};
\node at (2.6,9.55) {$u_1$};
\node at (4.6,5.55) {$v_2$};
\node at (2.6,5.55) {$u_2$};
\node at (4.6,1.55) {$v_3$};
\node at (2.6,1.55) {$u_3$};

{\color{red}\draw (x10) to[out=350,in=110] (z11);
\draw (x10) to[out=290,in=120] (z11);
\draw (x10) to[out=250,in=120] (z12);
\draw (x10) to[out=280,in=70] (z12);
\draw (x10) to[out=175,in=100] (z13);
\draw (x10) to[out=240,in=60] (z13);

\draw (z13) to[out=200,in=120] (w13);
\draw (z13) to[out=230,in=100] (w13);
\draw (z13) to[out=260,in=80] (w13);
\draw (z12) to[out=225,in=75] (w13);
\draw (z11) to[out=225,in=65] (w13);
\draw (z13) to[out=320,in=140] (w12);
\draw (z13) to[out=350,in=110] (w12);
\draw (z12) to[out=270,in=90] (w12);
\draw (z11) to[out=270,in=90] (w11);
\draw (z13) to[out=320,in=120] (w11);
\draw (z13) to[out=350,in=120] (w11);

\draw (w13) to[out=200,in=120] (y13);
\draw (w13) to[out=230,in=100] (y13);
\draw (w13) to[out=260,in=80] (y13);
\draw (w12) to[out=225,in=75] (y13);
\draw (w11) to[out=225,in=65] (y13);
\draw (w13) to[out=320,in=140] (y12);
\draw (w13) to[out=350,in=110] (y12);
\draw (w12) to[out=270,in=90] (y12);
\draw (w11) to[out=270,in=90] (y11);
\draw (w13) to[out=320,in=120] (y11);
\draw (w13) to[out=350,in=120] (y11);}

{\color{blue}
\draw(x'10) to[out=350,in=90] (z'11);
\draw (x'10) to[out=320,in=100] (z'11);
\draw (x'10) to[out=190,in=120] (z'12);
\draw (x'10) to[out=230,in=80] (z'12);

\draw (z'12) to[out=210,in=140] (w'12);
\draw (z'12) to[out=230,in=110] (w'12);
\draw (z'12) to[out=250,in=80] (w'12);
\draw (z'12) to[out=270,in=70] (w'12);
\draw (z'11) to[out=210,in=45] (w'12);
\draw (z'12) to[out=310,in=100] (w'11);

 \draw (w'12) to[out=210,in=140] (y'12);
\draw (w'12) to[out=230,in=110] (y'12);
\draw (w'12) to[out=250,in=80] (y'12);
\draw (w'12) to[out=270,in=70] (y'12);
\draw (w'11) to[out=210,in=45] (y'12);
\draw (w'12) to[out=310,in=100] (y'11);}

\node at (5.6,15) {$B_1$};
\node at (8.2,15) {$B_2$};
\draw [->](5,15) to (7,15);
\node at (6.9,15.3) {$f$};
\draw (x'10) to[out=30,in=155] (x10);
\draw (z13) to[out=95,in=50] (z'12);
\draw (z13) to[out=110,in=20] (z'12);
\draw (z12) to[out=140,in=20] (z'12);
\draw (z12) to[out=170,in=35] (z'11);
\draw (z11) to[out=120,in=30] (z'11);
\draw (z11) to[out=160,in=10] (z'12);

\draw (w13) to[out=95,in=50] (w'12);
\draw (w13) to[out=110,in=20] (w'12);
\draw (w12) to[out=140,in=20] (w'12);
\draw (w12) to[out=170,in=35] (w'11);
\draw (w11) to[out=120,in=30] (w'11);
\draw (w11) to[out=160,in=10] (w'12);
\draw (y13) to[out=95,in=50] (y'12);
\draw (y13) to[out=110,in=20] (y'12);
\draw (y12) to[out=140,in=20] (y'12);
\draw (y12) to[out=170,in=35] (y'11);
\draw (y11) to[out=120,in=30] (y'11);
\draw (y11) to[out=160,in=10] (y'12);


{\color{red}\draw (z13) ++(1.1,2) node {{\tiny $1\,........$}};
\draw (z13) ++(1.7,2) node {{\tiny $34$}};
\draw (z12) ++(0.6,2.2) node {{\tiny $1......$}};
\draw (z12) ++(1.05,2.2) node {{\tiny$21$}};
\draw (z11) ++(-0.1,2) node {{\tiny $1\, .....$}};
\draw (z11) ++(0.3,2) node {{\tiny $21$}};}
\draw (z12) ++(0.4,0.4) node {{\tiny $2$}};
\draw (z12) ++(0.2,0.2) node {{\tiny $1$}};
\draw (z11) ++(0.4,0.5) node {{\tiny $2$}};
\draw (z11) ++(0.2,0.2) node {{\tiny $1$}};
{\color{red}\draw (w13) ++(0.1,1.5) node {{\tiny $1$}};
\draw (w13) ++(0.45,1.5) node {{\tiny $2$}};
\draw (w13) ++(0.88,1.5) node {{\tiny $4$}};
\draw (w13) ++(1.2,1.5) node {{\tiny $3$}};
\draw (w13) ++(1.7,1.5) node {{\tiny $5$}};
\draw (w12) ++(0.1,1.1) node {{\tiny $1$}};
\draw (w12) ++(0.5,1.3) node {{\tiny $2$}};
\draw (w12) ++(0.81,1.4) node {{\tiny $3$}};
\draw (w11) ++(-0.58,1.8) node {{\tiny $1$}};
\draw (w11) ++(-0.3,1.95) node {{\tiny $2$}};
\draw (w11) ++(-0.8,2.3) node {{\tiny $3$}};}
\draw (w12) ++(0.4,0.4) node {{\tiny $1$}};
\draw (w12) ++(0.2,0.2) node {{\tiny $2$}};
\draw (w11) ++(0.4,0.5) node {{\tiny $1$}};
\draw (w11) ++(0.2,0.2) node {{\tiny $2$}};
{\color{red}\draw (y13) ++(0.1,1.5) node {{\tiny $2$}};
\draw (y13) ++(0.45,1.5) node {{\tiny $4$}};
\draw (y13) ++(0.891,1.5) node {{\tiny $5$}};
\draw (y13) ++(1.2,1.5) node {{\tiny $3$}};
\draw (y13) ++(1.7,1.5) node {{\tiny $1$}};
\draw (y12) ++(0.1,1.1) node {{\tiny $3$}};
\draw (y12) ++(0.5,1.3) node {{\tiny $2$}};
\draw (y12) ++(0.81,1.4) node {{\tiny $1$}};
\draw (y11) ++(-0.58,1.8) node {{\tiny $3$}};
\draw (y11) ++(-0.3,1.95) node {{\tiny $2$}};}
\draw (y11) ++(0.8,2.3) node {{\tiny $1$}};
\draw (y12) ++(0.4,0.4) node {{\tiny $2$}};
\draw (y12) ++(0.2,0.2) node {{\tiny $1$}};
\draw (y11) ++(0.4,0.5) node {{\tiny $2$}};
\draw (y11) ++(0.2,0.2) node {{\tiny $1$}};

{\color{blue}\draw (z'12) ++(0.55,2) node {{\tiny $1.....$}};
\draw (z'12) ++(0.97,2) node {{\tiny $17$}};
\draw (z'11) ++(0.8,2) node {{\tiny $1...$}};
\draw (z'11) ++(1.1,2) node {{\tiny $4$}};
\draw (w'12) ++(-0.15,1.7) node {{\tiny $1$}};
\draw (w'12) ++(0.25,1.7) node {{\tiny $2$}};
\draw (w'12) ++(0.75,1.7) node {{\tiny $3$}};
\draw (w'12) ++(1.05,1.6) node {{\tiny $4$}};
\draw (w'12) ++(1.2,1.5) node {{\tiny $5$}};
\draw (w'11) ++(0.25,1.7) node {{\tiny $1$}};
\draw (y'12) ++(-0.15,1.7) node {{\tiny $5$}};
\draw (y'12) ++(0.25,1.7) node {{\tiny $4$}};
\draw (y'12) ++(0.75,1.7) node {{\tiny $3$}};
\draw (y'12) ++(1.05,1.6) node {{\tiny $2$}};
\draw (y'12) ++(1.6,1.5) node {{\tiny $1$}};
\draw (y'11) ++(0.25,1.7) node {{\tiny $1$}};}
\end{tikzpicture}
\end{center}
\caption{}\label{1}
\end{figure}

Now one can determine the sequence of morphisms induced by $f$ between the two diagrams. To do that we first label the vertices of the two diagrams  by $V_i=\{u_i, v_i\}$ and $W_i=\{x_i, y_i, z_i\}$ for every $i\geq 1$. Then the morphisms are
\begin{eqnarray*}
\tau_i(x_i)=u_iu_i,\ \tau_i(y_i)=v_iu_i,\ \ \tau_i(z_i)=u_iv_i,  \ \ {\rm if}\ i\ {\rm is\ odd},\\
\tau_i(x_i)=u_iu_i,\ \tau_i(y_i)=u_iv_i,\ \ \tau_i(z_i)=v_iu_i,  \ \ {\rm if}\ i\ {\rm is\ even}.\\
\end{eqnarray*} 
Let us examine the order commutativity in terms of the morphisms on a fixed vertex of the right diagram. Consider for example  $y_2\in W_2$. Then
$$\tau_1\circ\sigma_2^{B_2}(y_2)=\tau_1(x_1x_1y_1)=\tau_1(x_1)\tau_1(x_1)\tau_1(y_1)=u_1u_1u_1u_1v_1u_1$$
and
$$\sigma_2^{B_1}\circ\tau_2(y_2)=\sigma_2^{B_1}(u_2v_2)=\sigma_2^{B_1}(u_2)\sigma_2^{B_1}(v_2)=u_1u_1u_1u_1v_1u_1.$$
\end{example}

%
We refer the reader to   \cite{AEG} to see more examples of ordered premorphisms.


\section{ Rank of Factors}
In this section we prove Theorem \ref{main}. To be prepared we firstly mention some observations and facts about {\it microscoping} of ordered Bratteli diagrams.

The following ``packing lemma'' is useful for reducing the rank
of an ordered Bratteli diagram.
Also, it is of interest in its own right.

\begin{lemma}\label{packing}
Let $B=((V_n)_{n\geq 0},\, (E_n)_{n\geq 1}
,\geq )$ be an ordered Bratteli diagram for which there exist some $k\geq 1$ 
and a set of words $W\subseteq V_{k-1}^*$ such that
\begin{equation}\label{equ_packing}
\sigma_{k}^B(v)\in W^*\ \ \text{for every}\
v\in V_{k}.
\end{equation}
Suppose  that $W$ is a minimal subset
of $V_{k-1}^*$
(with respect to the inclusion
relation) satisfying \eqref{equ_packing}.
Then there is an 
ordered Bratteli diagram $B'$ isomorphic to $B$ which is
constructed from $B$ 
by adding $W$ as a set of vertices between   levels
 $V_{k-1}$ and $V_{k}$.
\end{lemma}

\begin{proof}
Let
$W=\{w_1, \dots, w_s\}$ where $w_{i}$'s are distinct.
Define the set of vertices of
$B'=((V'_n)_{n\geq 0},\, (E'_n)_{n\geq 1}
,\geq' )$ by
\[
V'_n=V_n \ \ \text{for}\  \ 0\leq n < k,\ \
V'_k=W,\ \ \ \text{and}\ \ 
V'_n=V_{n-1}\ \  \text{for}\ n>k.
\]
For the set of edges of $B'$,
first we set 
\[
E'_n=E_n \ \ \text{for}\  \ 1\leq n < k,\ \
 \text{and}\ \ 
E'_n=E_{n-1}\ \  \text{for}\ n\geq k+2.
\]
It remains to define $E'_{k}$ and
$E'_{k+1}$. 
Since every $w\in W$ is a word
in $V_{k-1}^*$, we can define 
(uniquely)
a partially ordered set of edges  $E'_{k}$  from 
$V'_{k-1}=V_{k-1}$
to $V'_{k}=W$
such that $\sigma^{B'}_{k}(w)=w$,
for every $w\in W$.
To define $E'_{k+1}$, first note that
for every $v\in V'_{k+1}=V_k$
we have 
\begin{equation}\label{equ_packing_rep}
\sigma^{B}_{k}(v)= w_{i_{1}}w_{i_{2}}\cdots
w_{i_{r}},
\end{equation}
for some 
$w_{i_{1}},w_{i_{2}},\ldots,
w_{i_{r}}$ in $W$ depending on $v$.
(Note that this representation of 
$\sigma^{B}_{k}(v)$ in terms of
the words of $W$ is not necessarily unique
but we fix one representation.)
Then we can define
a partially ordered set of edges  $E'_{k+1}$  from 
$V'_{k}=W$
to $V'_{k+1}=V_k$
such that $\sigma^{B'}_{k+1}(v)=w_{i_{1}}w_{i_{2}}\cdots
w_{i_{r}}$, where 
$w_{i_{1}},w_{i_{2}},\ldots,
w_{i_{r}}$ are considered as vertices
of $V'_{k}$ here.
The minimality of $W$ guarantees that
for very $w\in W$ there is at least
one edge in  $E'_{k+1}$
with source $w$.
The resulting ordered Bratteli
diagram $B'$ is  
isomorphic to $B$ since
$E'_{k}\circ E'_{k+1}$ is order
isomorphic to $E_{k}$. In fact,
if $v\in V_{k}$ and if 
we consider the representation
in \eqref{equ_packing_rep} for
$\sigma^{B}_{k}(v)$, then
\[
\sigma^{B'}_{[k-1,k+1]}(v)=
\sigma^{B'}_{k}(\sigma^{B'}_{k+1}(v))
=\sigma^{B'}_{k}(w_{i_{1}}w_{i_{2}}\cdots
w_{i_{r}})=w_{i_{1}}w_{i_{2}}\cdots
w_{i_{r}}=\sigma^{B}_{k}(v).
\]
It follows that 
 $B'$ is  
isomorphic to $B$.
\end{proof}

\medskip
\begin{lemma}\label{basis}
Let $A$ be a finite alphabet and let $p\in \mathbb{N}$. 
Let $s_1, \dots, s_p, t_1, \dots, t_p$, and $w$ be words in $A^*$ such that 
\[
w=s_1t_1=s_2t_2=\cdots=s_pt_p.
\]
Suppose that there are two words $s$ and $t$ in $A^*$ with
$|s|,|t|\geq |w|$ such that for any  $1\leq i\leq p$, $s_i$ is a suffix of $s$ and $t_i$ is a prefix of $t$. 
Then there exists a set of words $B\subseteq A^*$ such that
\begin{enumerate}
\item\label{basis_it1}
$\mathrm{card}(B)\leq 3$;
\item\label{basis_it2}
$s_i, t_i\in B^*$
 for every 
$1\leq i\leq p$. 
\end{enumerate}
\end{lemma}

\begin{proof}
We may assume that each $s_i$ and each $t_i$ is nonempty.
Also, we may assume that $s_i \neq s_j$ for $i\neq j$.
If $p=1$ we take $B=\{s_{1},t_{1}\}$. Hence in the sequel we 
assume that $p\geq 2$. Note that for every $1\leq i, j\leq p$ since both $s_i$ and $s_j$ are suffixes of $s$, either $s_i$ is a suffix of $s_j$ or $s_j$ is a suffix of $s_i$.
Suppose that $s_i$'s are sorted by their lengths:
\begin{equation}\label{equ_sort}
|s_{1}|>|s_{2}|>\cdots >|s_{p}|.
\end{equation}
This implies that $|t_{1}|<|t_{2}|<\cdots <|t_{p}|$ (since $|s_i|+|t_i|=|w|$).
\begin{figure}
\begin{center}
\begin{tikzpicture}[scale=1.5]
\draw  (0,0) --(8.5,0);
\draw  (0,2) --(8.5,2);

\draw [thin] (4,-0.3) --(4,2.5);
\draw [dashed] (0.7,-0.15) --(0.7,2.15);
\draw [dashed] (2.2,-0.15) --(2.2,2.15);

\draw[thick] (0,2)  to [out=50,in=130] node[above]{$s$} (4,2);
\draw[thick] (4,2)  to [out=50,in=130] node[above]{$t$} (8.5,2);

\draw[thick] (0.7,2)  to [out=30,in=150] node[above]{$s_1$} (4,2);
\draw[thick] (4,2)  to [out=30,in=150] node[above]{$t_1$} (6,2);

\draw[thick] (0,0)  to [out=50,in=130] node[above]{$s$} (4,0);
\draw[thick] (4,0)  to [out=50,in=130] node[above]{$t$} (8.5,0);

\draw[thick] (2.2,0)  to [out=30,in=150] node[above]{$s_2$} (4,0);
\draw[thick] (4,0)  to [out=30,in=150] node[above]{$t_2$} (7.5,0);

\draw[thick] (0.7,2)  to [out=-30,in=-150] node[below]{$x$} (2.2,2);
\draw[thick] (2.2,2)  to [out=-30,in=-150] node[below]{$x$} (3.7,2);
\draw[thick] (3.7,2)  to [out=-30,in=-150] node[below]{$x$} (5.2,2);

\draw[thick] (5.2,2)  to [out=-30,in=-150] node[below]{$x'$} (6,2);

\draw[ultra thick] (0.7,2)  to (6,2);
\draw[ultra thick] (0.7,2.1)  to (0.7,1.9);
\draw[ultra thick] (6,2.1)  to (6,1.9);
\draw[->] (6,2.03) to (6.3,2.2);
\node at (6.4,2.3) {$w$};

\draw[ultra thick] (2.2,0)  to (7.5,0);
\draw[ultra thick] (2.2,0.1)  to (2.2,-0.1);
\draw[ultra thick] (7.5,0.1)  to (7.5,-0.1);
\draw[->] (7.5,0.03) to (7.8,0.2);
\node at (7.9,0.3) {$w$};

\draw[thick] (2.2,0)  to [out=-30,in=-150] node[below]{$x$} (3.7,0);
\draw[thick] (3.7,0)  to [out=-30,in=-150] node[below]{$x$} (5.2,0);
\draw[thick] (5.2,0)  to [out=-30,in=-150] node[below]{$x$} (6.7,0);
 
\draw[thick] (6.7,0)  to [out=-30,in=-150] node[below]{$x'$} (7.5,0);
\end{tikzpicture}
\end{center}
\caption{The two thick line segments represent the periodic word $w$.}\label{figure_st1}
\end{figure}
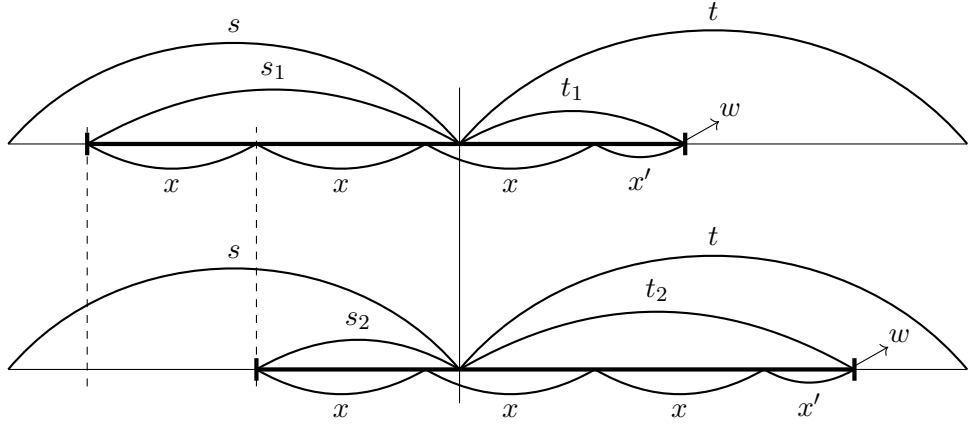

We claim that $w$ is periodic with period $|s_1|-|s_2|$.
To prove the claim, first note that both $s_{1}$ and $s_{2}$ are suffixes of  
$s$ and $|s_{1}|>|s_{2}|$. Hence there is a nonempty
word $x$ such that $s_1=xs_2$. 
See Figure~\ref{figure_st1}.

Similarly, since $t_1$ and $t_2$ are prefixes of $t$ and $|t_1|<|t_2|$,
there is a nonempty word $y$ such that $t_2=t_1 y$.
Note that $|x|<|s_1|<|w|$ and $|y|=|t_2|-|t_1|=|s_1|-|s_2|=|x|$.
We have
\begin{equation}\label{equ_wvuw}
wy=s_{1}t_{1}y=s_{1}t_{2}=xs_{2}t_{2}=xw.
\end{equation}
Let $|w|=k|x|+r$ where 
$k\in \mathbb{N}$ and $0\leq r<|x|$. It follows from \eqref{equ_wvuw} that
$wy^{k}=x^{k}w$. Since $|y^{k}|=|x^{k}|=k|x|\leq |w|$, we see that
$y^{k}$ is a suffix of $w$. So there is a word $x'$ such that 
$w=x'y^{k}$. Then $wy^{k}=x^{k}w=x^{k}x'y^{k}$ and so
$w=x^{k}x'$. Note that $|x'|=|w|-k|x|=r<|x|$. Moreover,
$x'y^{k}=w=x^{k}x'$ implies that $x'$ is a prefix of $x$.
Therefore, $w$ is periodic with period $|x|$. The claim is proved.

Using $w=s_i t_i=s_{i+1}t_{i+1}$, 
a similar argument shows that
$w$ is periodic with period
$|s_{i}|-|s_{i+1}|$, for all
$1\leq i <p$. We have
\[
\sum_{i=1}^{p-1}(|s_{i}|-|s_{i+1}|)=|s_1|-|s_p|<|w|.
\]
Applying Lemma~\ref{lem_FW}, it follows that $w$ is periodic with period 
\[
h=\gcd\big(|s_{1}|-|s_{2}|,\ldots,|s_{p-1}|-|s_{p}|\big).
\]
Thus we can write $w=u^{k}u'$ for some $u,u'\in A^{*}$ and $k\in \mathbb{N}$
where $|u'|<|u|=h$ and $u'$ is a prefix of $u$.
See Figure~\ref{figure_st2}.

Since each $s_{i}$ is a prefix of $w$, we can write $s_{i}=u^{k_{i}}u_{i}$
where $k_{i}\geq 0$, $|u_{i}|<|u|$, and $u_{i}$ is a prefix of $u$. 
We claim that $u_{1}=u_{2}=\cdots=u_{p}$. To see this, let $2\leq i\leq p$.
Since $s_{i}$ is a suffix of $s_1$ and $h$ divides $|s_1|-|s_i|$, we can
write $s_1=u^{\ell}s_{i}$ for some $\ell\geq 1$. Hence,
\[
u^{k_{1}}u_{1}=s_{1}=u^{\ell}s_{i}= u^{\ell}u^{k_{i}}u_{i}=
u^{\ell+k_{i}}u_{i}.
\]
Since $|u_{1}|,|u_{i}|<|u|$, this implies that $u_{1}=u_{i}$, and the claim 
is proved. Since $u_{1}$ is a prefix of $u$, there is a word
$u_{1}'$ such that $u=u_{1}u_{1}'$. See Figure~\ref{figure_st2}.

\begin{figure}
\begin{center}
\begin{tikzpicture}[scale=1.5]
\draw  (0,0) --(8.5,0);
\draw  (0,2) --(8.5,2);
\draw  (0,-4) --(8.5,-4);

\draw [thin] (4,-4.3) --(4,2.5);

\draw[ultra thick] (0.7,2)  to (6,2);
\draw[ultra thick] (0.7,2.1)  to (0.7,1.9);
\draw[ultra thick] (6,2.1)  to (6,1.9);
\draw[->] (6,2.03) to (6.3,2.2);
\node at (6.4,2.3) {$w$};

\draw[ultra thick] (2.2,0)  to (7.5,0);
\draw[ultra thick] (2.2,0.1)  to (2.2,-0.1);
\draw[ultra thick] (7.5,0.1)  to (7.5,-0.1);
\draw[->] (7.5,0.03) to (7.8,0.2);
\node at (7.9,0.3) {$w$};

\draw[ultra thick] (2.7,-4)  to (8,-4);
\draw[ultra thick] (2.7,-3.9)  to (2.7,-4.1);
\draw[ultra thick] (8,-3.9)  to (8,-4.1);
\draw[->] (2.7,-3.97) to (2.4,-3.8);
\node at (2.3,-3.7) {$w$};

\draw[thick] (0,2)  to [out=50,in=130] node[above]{$s$} (4,2);
\draw[thick] (4,2)  to [out=50,in=130] node[above]{$t$} (8.5,2);

\draw[thick] (0.7,2)  to [out=30,in=150] node[above]{$s_1$} (4,2);
\draw[thick] (4,2)  to [out=30,in=150] node[above]{$t_1$} (6,2);

\draw[thick] (0,0)  to [out=50,in=130] node[above]{$s$} (4,0);
\draw[thick] (4,0)  to [out=50,in=130] node[above]{$t$} (8.5,0);

\draw[thick] (2.2,0)  to [out=30,in=150] node[above]{$s_2$} (4,0);
\draw[thick] (4,0)  to [out=30,in=150] node[above]{$t_2$} (7.5,0);

\filldraw (2.2,-1.5) circle [radius=0.01];
\filldraw (2.2,-1.7) circle [radius=0.01];
\filldraw (2.2,-1.9) circle [radius=0.01];
\filldraw (6.2,-1.5) circle [radius=0.01];
\filldraw (6.2,-1.7) circle [radius=0.01];
\filldraw (6.2,-1.9) circle [radius=0.01];

\draw[thick] (0,-4)  to [out=50,in=130] node[above]{$s$} (4,-4);
\draw[thick] (4,-4)  to [out=50,in=130] node[above]{$t$} (8.5,-4);

\draw[thick] (2.7,-4)  to [out=30,in=150] node[above]{$s_p$} (4,-4);
\draw[thick] (4,-4)  to [out=30,in=150] node[above]{$t_p$} (8,-4);

\draw[thick] (0.7,2)  to [out=-30,in=-150] node[below]{$u$} (1.2,2);
\draw[thick] (1.2,2)  to [out=-30,in=-150] node[below]{$u$} (1.7,2);
\draw[thick] (1.7,2)  to [out=-30,in=-150] node[below]{$u$} (2.2,2);
\draw[thick] (2.2,2)  to [out=-30,in=-150] node[below]{$u$} (2.7,2);
\draw[thick] (2.7,2)  to [out=-30,in=-150] node[below]{$u$} (3.2,2);
\draw[thick] (3.2,2)  to [out=-30,in=-150] node[below]{$u$} (3.7,2);
\draw[thick] (3.7,2)  to [out=-30,in=-150] node[below]{$u$} (4.2,2);
\draw[thick] (4.2,2)  to [out=-30,in=-150] node[below]{$u$} (4.7,2);
\draw[thick] (4.7,2)  to [out=-30,in=-150] node[below]{$u$} (5.2,2);
\draw[thick] (5.2,2)  to [out=-30,in=-150] node[below]{$u$} (5.7,2);

\draw[thick] (5.7,2)  to [out=-40,in=-140] node[below]{\small{$u'$}} (6,2);

\draw[thick] (2.2,0)  to [out=-30,in=-150] node[below]{$u$} (2.7,0);
\draw[thick] (2.7,0)  to [out=-30,in=-150] node[below]{$u$} (3.2,0);
\draw[thick] (3.2,0)  to [out=-30,in=-150] node[below]{$u$} (3.7,0);
\draw[thick] (3.7,0)  to [out=-30,in=-150] node[below]{$u$} (4.2,0);
\draw[thick] (4.2,0)  to [out=-30,in=-150] node[below]{$u$} (4.7,0);
\draw[thick] (4.7,0)  to [out=-30,in=-150] node[below]{$u$} (5.2,0);
\draw[thick] (5.2,0)  to [out=-30,in=-150] node[below]{$u$} (5.7,0);
\draw[thick] (5.7,0)  to [out=-30,in=-150] node[below]{$u$} (6.2,0);
\draw[thick] (6.2,0)  to [out=-30,in=-150] node[below]{$u$} (6.7,0);
\draw[thick] (6.7,0)  to [out=-30,in=-150] node[below]{$u$} (7.2,0);

\draw[thick] (7.2,0)  to [out=-40,in=-140] node[below]{\small{$u'$}} (7.5,0);

\draw[thick] (2.7,-4)  to [out=-30,in=-150] node[below]{$u$} (3.2,-4);
\draw[thick] (3.2,-4)  to [out=-30,in=-150] node[below]{$u$} (3.7,-4);
\draw[thick] (3.7,-4)  to [out=-30,in=-150] node[below]{$u$} (4.2,-4);
\draw[thick] (4.2,-4)  to [out=-30,in=-150] node[below]{$u$} (4.7,-4);
\draw[thick] (4.7,-4)  to [out=-30,in=-150] node[below]{$u$} (5.2,-4);
\draw[thick] (5.2,-4)  to [out=-30,in=-150] node[below]{$u$} (5.7,-4);
\draw[thick] (5.7,-4)  to [out=-30,in=-150] node[below]{$u$} (6.2,-4);
\draw[thick] (6.2,-4)  to [out=-30,in=-150] node[below]{$u$} (6.7,-4);
\draw[thick] (6.7,-4)  to [out=-30,in=-150] node[below]{$u$} (7.2,-4);
\draw[thick] (7.2,-4)  to [out=-30,in=-150] node[below]{$u$} (7.7,-4);

\draw[thick] (7.7,-4)  to [out=-40,in=-140] node[below]{\small{$u'$}} (8,-4);

\draw[thick, fill] (3.7,2)  to [out=80,in=100] node[above]{\small{$u_1$}} (4,2);
\draw[thick, fill] (4,2)  to [out=80,in=100] node[above]{\small{$u'_1$}} (4.2,2);
\end{tikzpicture}
\end{center}
\caption{}\label{figure_st2}
\end{figure}

We consider the following three cases:

\bigskip

\textbf{Case~I:} $u_{1}$ is the empty word. Then
$s_{i}=u^{k_{i}}$, and hence $t_{i}=u^{\ell_{i}}u'$ for some $\ell_{i}\geq 0$.
We set $B=\{u,u'\}\setminus \{\emptyset\}$.
(Note that $u'$ may be the empty word.)
Then $B$ generates all $s_{i}$'s and $t_{i}$'s.

\medskip

\textbf{Case~II:} $u_{1}$ is nonempty and $|u_{1}'|>|t_{1}|$.
This is equivalent to $|u_1t_1|<|u|$,
since $u=u_{1}u_{1}'$.
 Then $u'=u_1t_1$ and for every $2\leq i\leq p$, 
since $t_1$ is prefix of $t_i$ there exists some $\ell_i\geq 0$ such that 
$t_i=u'_1u^{\ell_i}u'$. All together imply that $B=\{u_1, u'_1, t_1\}$ 
generates all $s_{i}$'s and $t_{i}$'s.

\medskip

\textbf{Case~III:} $u_{1}$ is nonempty and $|u_{1}'|\leq |t_{1}|$.
This means that $|u_1 t_1|>|u|$. 
It follows that $t_1=u'_1 u^{\ell_1}u'$ for some $\ell_1\geq 0$, and  for every 
$2\leq i\leq p$, there exists $\ell_i\geq 1$ such that $t_i=u'_1u^{\ell_i}u'$. 
Hence, $B=\{u_1, u'_1, u'\}\setminus\{\emptyset\}$
generates all $s_{i}$'s and $t_{i}$'s. 

\bigskip

Therefore, in each case we obtained a set of words $B$ satisfying
Conditions~\eqref{basis_it1} and \eqref{basis_it2}.
This finishes the proof.
\end{proof}

\begin{remark}
In the preceding lemma,
the upper bound 3 for the cardinal of the set $B$ is sharp.
For example, let $A=\{x,y,z,w\}$, 
$w=xyzwxyzwxyzwx$, and
\[
s_1=xyzwxyzwxy,\ t_1=zwx,\ s_2=xyzwxy,\ t_2=zwxyzwx, 
\]
\[ s_3=xy, \ 
t_3=zwxyzwxyzwx,\ 
s=s_1,\ t=t_3.
\]
Then the argument in the proof of the preceding lemma
gives the set of words $B=\{u_{1},u_{1}', u'\}=\{xy, zw, x\}$  
 generating all $s_i$'s and
$t_i$'s. However, it is not hard to see that there is
no generating set $B'$ with $\mathrm{card}(B')<3$.
\end{remark}

\begin{remark}\label{rmk_proper}
Let $B$   be a
properly  ordered Bratteli diagram. Then there exits
a telescoping of $B$, say
$B'=((V_k)_{k\geq 0},\, (E_k)_{k\geq 1},\geq)$, 
such that for each $k\geq 0$ there are 
(necessarily unique) vertices $v_{\min}^{k}$ and  $v_{\max}^{k}$
in $V_{k}$ such that for every $v\in V_{k+1}$,
$\sigma_{k+1}^{B'}(v)$ starts with $v_{\min}^{k}$ and  
ends with $v_{\max}^{k}$, that is, the min edge in $E_{k+1}$
to $v$ comes from $v_{\min}^{k}$ and the max edge to $v$
comes from $v_{\max}^{k}$. This simple fact follows 
easily from an argument using the K\"{o}nig's lemma similar to
the argument showing that every ordered Bratteli diagram
has at least one min infinite path.
\end{remark}

\begin{proposition}\label{engin}
Let $f\colon B_{1}\rightarrow B_2$ be an ordered  premorphism 
between two  properly ordered Bratteli diagrams such that  $B_1$ is simple. 
Consider the the Vershik system on $B_{1}$. Then
\[
{\rm rank}_{\rm top}(X_{B_1},\, T_{B_1})\leq 3\,  {\rm rank}(B_2).
\]
\end{proposition}

\begin{proof}
Let $B_1=(V,E,\geq)$, $B_{2}=(W,S,\geq)$, and 
$f=(F, (f_{n})_{n=0}^{\infty},\geq)$. 
Let $V=\cup_{n=0}^{\infty}V_{n}$ and 
$W=\cup_{n=0}^{\infty}W_{n}$
be the canonical decompositions of $V$ and $W$, respectively.
Also, let the morphisms $\tau_n$'s associated to the ordered premorphism
$f$ be as in Definition~\ref{new}.

If ${\rm rank}_{\rm top}(X_{B_1},\, T_{B_1})=1$ then there is nothing
to prove. Thus we suppose that 
${\rm rank}_{\rm top}(X_{B_1},\, T_{B_1})\geq 2$. In particular,
$B_1$ has infinitely many levels each of which having at least
two vertices.

By making telescopings of the two diagrams along appropriate 
subsequences if necessary, we can (and do) make these assumptions:
\begin{enumerate}
\item\label{engin_it1}
For every $n\in \mathbb{N}$, every $v\in V_{n}$,
and every $v'\in V_{n+1}$, there is an edge
in $E_{n+1}$ with source $v$ and range $v'$
(since $B_1$ is simple).

\item\label{engin_it2}
For every $n\geq 0$ there are 
vertices $v_{\min}^{n}$ and  $v_{\max}^{n}$
in $V_{n}$ such that for every $v\in V_{n+1}$,
$\sigma_{n+1}^{B_1}(v)$ starts with $v_{\min}^{n}$ and  
ends up with $v_{\max}^{n}$
(by Remark~\ref{rmk_proper}).

\item\label{engin_it3}
$\mathrm{card}(V_{n})\geq 2$, for all $n\in \mathbb{N}$
(by the preceding paragraph).

\item\label{engin_it4}
For every $n\geq 0$, $f_n=n$, and
for every $v\in V_{n}$ and every $w\in W_{n}$ there is an edge
in $F$ 
with source $v$ and range $w$.
(This follows from \eqref{engin_it1} and Definition \ref{new}),
and  an appropriate telescoping of $B_2$.)
\setcounter{TmpEnumi}{\value{enumi}}
\end{enumerate}


The following claim contains the main
part of the proof. 
\smallskip

{\bf Claim.} For every $n\geq 1$ there exist
 some $\ell>n$ and a set of words
  $C_n\subseteq V_n^*$ such that 
\begin{enumerate} 
\setcounter{enumi}{\value{TmpEnumi}}
\item\label{engin_it5} 
$\sigma_{[n,\ell]}^{B_1}(v)\in C_n^*$, for all $v\in V_\ell$;
\item\label{engin_it6} 
${\rm card}(C_n)\leq 3\, {\rm card}(W_n)$.
\setcounter{TmpEnumi}{\value{enumi}}
\end{enumerate} 
To prove the claim,
fix $n\geq 1$.
First, using \eqref{engin_it1} and
\eqref{engin_it3},  there is $\ell>n+1$ 
such that
\begin{enumerate}
\setcounter{enumi}{\value{TmpEnumi}}
\item\label{engin_it7}
 $\big|\sigma_{[n,\ell-1]}^{B_1}(v_{\min}^{\ell-1})\big|,\,
 \big|\sigma_{[n,\ell-1]}^{B_1}(v_{\max}^{\ell-1})\big|
 > \max\{|\tau_n(w)| \colon \ w\in W_n\}$.
 \setcounter{TmpEnumi}{\value{enumi}}
\end{enumerate}

 Fix an arbitrary vertex
 $w_0\in W_{\ell}$. Suppose that  
 $\tau_{\ell}(w_0)=v_1 v_2\cdots v_m$,
 where
 $v_1, v_2, \ldots,v_m \in V_{\ell}$.
 By \eqref{engin_it4}, every
 vertex of $V_{\ell}$ appears at least
 one time in the word $\tau_{\ell}(w_0)$,
 and hence
 $ V_{\ell}=\{v_{1},v_2, \ldots,v_m\}$
 as sets. As $\mathrm{card}(V_{\ell})\geq 2$, we see that $m\geq 2$.
Suppose that
 $\sigma_{[n,\ell]}^{B_2}(w_{0})=
 w_1 w_2 \cdots w_r$
 where 
 $w_1, w_2, \ldots,w_r \in W_{n}$.
 Set $z_j=\tau_n(w_{j})$ for all
 $1\leq j\leq r$.
 Using the ordered commutativity of $f$
at the second step (see equality \eqref{equ_ordcom_mor}),
we get
\begin{eqnarray}\label{equ_engin_ordcom}
 \sigma_{[n,\ell]}^{B_1}(v_1)
 \sigma_{[n,\ell]}^{B_1}(v_2)\cdots
 \sigma_{[n,\ell]}^{B_1}(v_m) &=&
 \sigma_{[n,\ell]}^{B_1}(\tau_{\ell}(w_0))
 \notag\\
 &=&\tau_{n}(\sigma_{[n,\ell]}^{B_2}(w_{0}))\notag\\
 &=&\tau_{n}(w_1  w_2  \cdots w_r)\\
 &=& z_1  z_2  \cdots z_r.\notag
\end{eqnarray}
For all $1\leq k\leq m$,
by \eqref{engin_it2},
the word
 $ \sigma_{\ell}^{B_1}(v_k)$
 starts with $v_{\min}^{\ell-1}$
 and ends up with $v_{\max}^{\ell-1}$.
Hence, for all $1\leq k\leq m$,
\begin{enumerate}
\setcounter{enumi}{\value{TmpEnumi}}
\item\label{engin_it8}
$\sigma_{[n,\ell]}^{B_1}(v_k)$
 starts with
$ \sigma_{[n,\ell-1]}^{B_1}(v_{\min}^{\ell-1})$ and ends up with
$\sigma_{[n,\ell-1]}^{B_1}(v_{\max}^{\ell-1})$.
\setcounter{TmpEnumi}{\value{enumi}}
\end{enumerate} 
In particular,   
by \eqref{engin_it7}, 
$ \big|\sigma_{[n,\ell]}^{B_1}(v_k)\big|>
|z_j|$, for all $1\leq j\leq r$
 and all $1\leq k\leq m$.
 Using this and \eqref{equ_engin_ordcom},
 it follows that there are words
 $s_1,s_2,\ldots,s_{m-1}$ and
 $t_1,t_2,\ldots,t_{m-1}$
 in $V_{n}^{*}$
 and words
 $T_1,T_2,\ldots,T_{m}$ in
 $\{z_1, \dots, z_{r}\}^*$
 such that

\begin{equation}\label{equ_engin_st}
\sigma_{[n,\ell]}^{B_1}(v_1)=T_1s_1,\ \ \sigma_{[n,\ell]}^{B_1}(v_k)=t_{k-1}T_k s_k, \ \ \sigma_{[n,\ell]}^{B_1}(v_m)=t_{m-1}T_m,
\end{equation}
 for all $1<k<m$,
where
  for every $i=1, \dots, m-1$,
  $s_i=t_i=\emptyset$ or
   both 
  $s_i$ and $t_i$ are nonempty
  and $s_it_i=z_{j}$
   for some $1\leq j\leq r$. 
This means that we can write
\[
\sigma_{[n ,\ell]}^{B_1}(\tau_\ell(w_{0}))=T_1s_1t_1T_2s_2t_2T_3s_3\cdots s_{m-1}t_{m-1}T_m=z_1  z_2  \cdots z_r.
\]

 Now we want to analyze $s_i$'s and $t_i$'s to see how one can generate all of them by a set of words $C_{n}$  of cardinal less than or  equal to 
 $3\,  \mathrm{card}(W_n)$. Set
 \[
s=\sigma_{[n,\ell-1]}^{B_1}(v_{\max}^{\ell-1})\
\ \ \text{and}\ \ \
t=\sigma_{[n,\ell-1]}^{B_1}(v_{\min}^{\ell-1}).
 \]
 By \eqref{engin_it7},
 $|s|,|t|> |z_j|$, for all $j=1,\ldots,r$.
 Also, by \eqref{engin_it8},
 every $\sigma_{[n,\ell]}^{B_1}(v_k)$
 starts with
$ t$ and ends up with
$s$.
Then \eqref{equ_engin_st} implies
that each
  $s_i$ is a proper suffix of $s$ and
  each $t_i$ is a proper prefix of $t$.  
Note that we have
\begin{equation*}
\big\{s_i t_i \colon 1\leq i < m \big\}
\setminus \{\emptyset\}
\subseteq
\big\{z_j \colon 1\leq j \leq r \big\}
\subseteq
\big\{\tau_n(w) \colon \ w\in W_n \big\}.
\end{equation*}
Let $w\in W_{n}$. If there is no $i$
with $s_i t_i =\tau_{n}(w)$
then we simply set
$C_{w}=\{\tau_{n}(w)\}$.
Otherwise, 
we apply Lemma~\ref{basis}
with $\tau_{n}(w)$ in place of $w$,
with $s$ and $t$ as above, and
with $s_i$'s and $t_i$'s
satisfying $s_i t_i =\tau_{n}(w)$
to obtain a set of words 
$C_w\subseteq V_{n}^{*}$ such that
\begin{enumerate}
\setcounter{enumi}{\value{TmpEnumi}}
\item\label{engin_it9}
$\mathrm{card}(C_w)\leq 3$;
\item\label{engin_it10}
$s_i, t_i\in C_{w}^*$
 for every 
$1\leq i< m$ with $s_i t_i= \tau_{n}(w)$.
\setcounter{TmpEnumi}{\value{enumi}}
\end{enumerate}

Now consider the set of words
\[
C_n=\cup_{w\in W_n}C_w \subseteq V_{n}^{*}.
\]
First note that $\tau_{n}(w)\in C_{n}^{*}$
for every $w\in W_{n}$, and
hence 
$z_1,\ldots,z_r\in C_{n}^{*}$.
Moreover, $s_i,t_i \in C_{n}^{*}$ 
for every $1\leq i<m$, because
if $s_i$ and $t_i$ are nonempty
then there is some $j$ with 
$s_i t_i=z_j=\tau_{n}(w_j)$
and hence $s_i, t_i\in C_{w_j}^*$ 
by \eqref{engin_it10}.
Therefore, 
$\sigma_{[n,\ell]}^{B_1}(v_k)\in C_{n}^{*}$
for all $1\leq k\leq m$.
Since  $ V_{\ell}=\{v_{1},v_2, \ldots,v_m\}$, this implies that 
$\sigma_{[n,\ell]}^{B_1}(v)\in C_{n}^{*}$
for all $v\in V_{\ell}$.
This is \eqref{engin_it5}.
Also, \eqref{engin_it6} follows
from
\eqref{engin_it9}.
This finishes the proof of the claim.

\medskip

To complete the proof of the proposition, 
we will find an
ordered Bratteli diagram
$B_{1}''$
equivalent to $B_1$
such that $\mathrm{rank}(B_{1}'')\leq3\,\mathrm{rank}(B_{2})$.
For this, first we put
$\ell_1=1$ and we apply the claim
with $n=\ell_1$ to obtain a natural number
 $\ell_2>\ell_1$ and a set of words
  $C_{\ell_1}\subseteq V_{\ell_1}^*$ such that 
$\sigma_{[\ell_1,\ell_2]}^{B_1}(v)\in C_{\ell_1}^*$, for all $v\in V_{\ell_2}$, and 
${\rm card}(C_{\ell_1})\leq 3\, {\rm card}(W_{\ell_1})$.
Then we apply the claim with $n=\ell_2$
to obtain $\ell_3>\ell_2$ and 
$C_{\ell_{2}}$.
Continuing this procedure,
we obtain a strictly increasing
sequence
$(\ell_k)_{k=1}^{\infty}$
and a sequence $(C_{\ell_{k}})_{k=1}^{\infty}$
such that $C_{\ell_{k}}\subseteq V_{\ell_{k}}^{*}$, 
$\sigma_{[\ell_k,\ell_{k+1}]}^{B_1}(V_{\ell_{k+1}})\subseteq C_{\ell_k}^*$, 
and 
${\rm card}(C_{\ell_k})\leq 3\, {\rm card}(W_{\ell_k})$.
By passing to a subset of
$C_{\ell_{k}}$
 if necessary, 
 may assume that 
$C_{\ell_{k}}$ is a minimal subset
of $ V_{\ell_{k}}^{*}$
(with respect to the inclusion relation)
with the property
$\sigma^{B_1}_{[\ell_k, \ell_{k+1}]}(V_{\ell_{k+1}})\subset C_{\ell_{k}}^*$.

Now define an ordered
Bratteli diagram $B'_1=(V',E',\geq')$
(which will be a microscoping of a telescoping
of $B_1$) 
as follows. Set $\ell_{0}=0$ and
$C_{\ell_{0}}=V_0$.
The vertices of $B_1'$ are defined by
\[
V'_{2k}=C_{\ell_{k}}\ \ \
\text{and}\ \ \ 
V'_{2k+1}=V_{\ell_{k+1}},\ \ k\geq 0.
\]
For the set $E'=\cup_{k=1}^{\infty}E'_{k}$
 of edges of $B_{1}'$, first
 we set $E'_{1}=E_{1}$.
Next, let $k\geq 1$. 
Applying Lemma~\ref{packing} to 
the telescoping of $B_1$ along the
sequence 
\[
0,1,2,\ldots, \ell_{k},\ell_{k+1},
\ell_{k+1}+1,\ell_{k+1}+2,\ldots,
\]
we obtain two sets
$E'_{2k}$ and $E'_{2k+1}$
of edges
from $V_{\ell_{k}}$ to 
$C_{\ell_{k}}$ and
from $C_{\ell_{k}}$ to
 $V_{\ell_{k+1}}$,
such that
 $E'_{2k}\circ E'_{2k+1}$
 is order isomorphic to $E_{\ell_{k},\ell_{k+1}}$
 (see  the proof of Lemma~\ref{packing}).
 Thus 
$E'_{2k}$ is a set of edges from 
$V'_{2{k}-1}$ to $V'_{2{k}}$,
and $E'_{2k+1}$ is a set of edges from 
$V'_{2{k}}$ to $V'_{2{k}+1}$.
In this way, we obtain
an ordered Bratteli diagram
 $B'_1=(V',E',\geq')$.

Let $B_1'' =(V'',E'',\geq'')$
be the telescoping of 
 $B'_1$ along the even levels.
 Thus $V_{k}''=C_{\ell_{k}}$ for every
 $k\geq 0$.
 Since the telescoping of  $B'_1$ along the odd levels is to isomorphic 
 the telescoping of $B_{1}$ along
 the sequence $(\ell_{k})_{k=0}^{\infty}$,
 it follows that  $B'_1$
 is  simple and properly ordered,
 and hence so is $B_{1}''$.
 Moreover, $B_{1}$, $B_{1}'$, and
 $B_{1}''$ are  isomorphic as two ordered Bratteli diagrams,
 and so their associated 
 Vershik systems are conjugate.
Therefore,
\begin{eqnarray*}
{\rm rank}_{\rm top}(X_{B_1},\, T_{B_1})&=&
{\rm rank}_{\rm top}(X_{B_1''},\, T_{B_1''})\\
&\leq&
{\rm rank}(B_1'')\\ &=&
\sup\big\{\mathrm{card}(C_{\ell_{k}})\colon
k\geq 0\big\}\\
&\leq &
\sup\big\{3\,\mathrm{card}(W_{\ell_{k}})\colon
k\geq 0\big\}\\
&\leq &
 3\,  {\rm rank}(B_2).
\end{eqnarray*}
This finishes the proof.
\end{proof}

Now we have the requires for the proof of our main result.

\begin{proof}[Proof of Theorem \ref{main}]
Choose a point $x\in X$ so that the properly ordered Bratteli diagram $B_1$ associated to $(X,\,T)$ with base on $x_{\min}:=x$  realizes the topological rank of $(X,\,T)$, that is ${\rm rank}_{\rm top}(X,\, T)={\rm rank}(B_1)$. Suppose that $y:=\alpha(x)\in Y$ and  let $B_2$ be  the simple properly ordered Bratteli diagram associated to $(Y,\,S)$ based on the point $y_{\min}:=y$. Then by Proposition \ref{AEG_prop} there exists an ordered premorphism $f:B_2\rightarrow B_1$. Using the conjugacy of $(Y,\, S)$ and $(X_{B_2},\, T_{B_2})$ and  Proposition \ref{engin} one can conclude that
 $${\rm rank}_{\rm top}(Y,\,S)={\rm rank}_{\rm top}(X_{B_2},\, T_{B_2})\leq 3 \ {\rm rank}(B_1)=3\ {\rm rank}(X, T)$$
 as desired.
\end{proof}

\section{Conjugacy}
In this section we obtain a combinatorial criterion 
to verify conjugacy of two Cantor minimal systems.
Suppose that $(X,T)$ and $(Y,S)$ are
Cantor minimal systems
(or, more generally, essentially minimal
systems) for which there are Bratteli-Vershik 
models $B_2$ and $B_1$, respectively, with an ordered
premorphism $f\colon B_1\to B_2$.
This gives a factor map $\alpha\colon X\to Y$,
by Proposition~\ref{AEG_prop}. 
However, 
if $f$ satisfies the conditions of the following proposition
then $B_1$ is equivalent to $B_2$ and hence
$(X,T)$ is conjugate to $(Y,S)$.

We need the following notion in the sequel.

\begin{definition}\label{def_ind}
Let $A$ be an alphabet. We say that
a set of words $B\subseteq A^{*}$
is \emph{independent}
if every word in $B^{*}$
has a unique representation
in terms of the elements of $B$.
\end{definition}

For example, if no word in $B$ is
a prefix of another word, then 
$B$ is independent.

\begin{proposition}\label{prop_conj}
Let $f\colon B_1\to B_2$ be an 
ordered premorphism between two 
ordered Bratteli diagrams. Assume the notation in 
Definitions~\ref{defpreobd} and \ref{new}.
Suppose that for infinitely many
$n\in \mathbb{N}$ the followings hold:
\begin{enumerate}
\item\label{prop_conj_it1}
the set $D_n =\{\tau_n(w) : w\in W_{f_{n}}\}\subseteq V_{n}^{*}$
is independent;

\item\label{prop_conj_it2}
there is $\ell>n$ such that
$\sigma_{[n,\ell]}^{B_1}(v)\in D_n^*$, for all $v\in V_\ell$, and 
$D_n$ is a minimal subset of $V_{n}^{*}$
(with respect to inclusion)
having this property.
\end{enumerate}
Then $B_1$ is equivalent to $B_2$.
\setcounter{TmpEnumi}{\value{enumi}}
\end{proposition}

\begin{proof}
By passing to appropriate telescopings of $B_1$ and
$B_2$,
we may assume that
$f_n=n$ for any $n\geq 0$,
 \eqref{prop_conj_it1}
and \eqref{prop_conj_it2} hold
for every $n\in \mathbb{N}$,
and $\ell=n+1$ in \eqref{prop_conj_it2}.
To prove that  $B_1$ is equivalent to $B_2$
we will construct an ordered Bratteli diagram $B$ whose telescoping along
the odd (respectively, even) levels is equivalent to $B_1$
(respectively, $B_2$).

In the sequel, by an ordered set $G_n$ of edges from $W_n$
to $V_{n+1}$ we mean a finite nonempty set of edges
with a pair of source and range maps
$s\colon G_n\to W_{n} $  and 
$r\colon G_n\to V_{n+1}$, respectively, and with
a partial ordering on $G_n$ such that
$g,g'\in G_n$ are comparable if and only if
$r(g)=r(g')$,  the restriction of this ordering
on each $r^{-1}\{v\}$, $v\in V_{n+1}$, is a linear ordering,
and $r^{-1}\{v\}$ and $s^{-1}\{w\}$
are nonempty for all  $v\in V_{n+1}$ and $w\in W_{n}$.

\medskip

We claim that for every $n\in\mathbb{N}$
there is an ordered set 
$G_n$ of edges from $W_n$
to $V_{n+1}$ such that

\begin{enumerate}
\setcounter{enumi}{\value{TmpEnumi}}
\item\label{prop_conj_it3}
$F_n\circ G_n$ is order isomorphic
to $E_n$;

\item\label{prop_conj_it4}
$G_n\circ F_{n+1}$ is order isomorphic
to $S_n$;

\item\label{prop_conj_it5}
for every $w\in W_{n}$ and every
$v\in V_{n+1}$ there are $g,g'\in G_{n}$
with $s(g)=w$ and $r(g')=v$.
\end{enumerate}
In fact, \eqref{prop_conj_it3}
and \eqref{prop_conj_it4} say that in the following diagram, the two triangles are ordered commutative.

\[
\xymatrix{V_{n}\ar[r]^{F_{n}}\ar[d]_{E_{n}}
 & W_{n}\ar[d]^{S_{n}} \ar[ld]_{G_{n}} \\
 V_{n+1} \ar[r]_{F_{n+1}}
 & W_{n+1}
 }
\]
\medskip

To prove the claim, first note that by 
\eqref{prop_conj_it1} and \eqref{prop_conj_it2}
(recall that $\ell=n+1$ in \eqref{prop_conj_it2}), for every 
$v\in V_{n+1}$ the word
$\sigma_{n+1}^{B_1}(v)$
has a unique representation 
with respect to  $D_{n}$ which is
\begin{equation}\label{equ1_prop_conj}
\sigma_{n+1}^{B_1}(v)=\tau_{n}(w_{i_{1}})
\tau_{n}(w_{i_{2}})\cdots
\tau_{n}(w_{i_{k}}),
\end{equation}
for some $w_{i_{1}},w_{i_{2}},\ldots
w_{i_{k}}$ in $W_n$.
(Note that, by the independence of $D_n$, 
we have  implicitly assumed that
$\tau_n(w)\neq \tau_n(w')$
for $ w\neq  w'$.) Hence, there is
a (necessarily unique) ordered set of 
edges $G_n$ from $W_n$ to $V_{n+1}$
that induces a morphism
$\sigma_n\colon V_{n+1}\to W_{n}^{*}$
such that for every $v\in V_{n+1}$ (by
\eqref{equ1_prop_conj}), 
\begin{equation}\label{equ2_prop_conj}
\sigma_{n}(v)= w_{i_{1}}
w_{i_{2}}\cdots
w_{i_{k}}.
\end{equation}
Moreover, since  $D_n$ is  minimal, for every $w\in W_n$
there exists some $v\in V_{n+1}$ such that $w$ occurs as a letter of
$\sigma_{n}(v)$. Thus \eqref{prop_conj_it5}
holds.

To see \eqref{prop_conj_it3}, it is enough
to show that the morphisms
$\tau_n \circ \sigma_n$ and 
$\sigma_{n+1}^{B_1}$ (associated to
$F_n \circ G_n$ and $E_n$, respectively)
are the same. Let $v\in V_{n+1}$ satisfy
\eqref{equ1_prop_conj}. Then
\[
\sigma_{n+1}^{B_1}(v)=\tau_{n}(w_{i_{1}})
 \cdots
\tau_{n}(w_{i_{k}})=
\tau_{n}(w_{i_{1}}
 \cdots
w_{i_{k}})=
\tau_{n}(\sigma_n(v)).
\]
Thus 
$\sigma_{n+1}^{B_1}=\tau_n \circ \sigma_n$,
proving \eqref{prop_conj_it3}. For
\eqref{prop_conj_it4}, first by
\eqref{prop_conj_it3} and the ordered
commutativity of $f$ we have
\begin{equation}\label{equ3_prop_conj}
F_n \circ G_n  \circ F_{n+1}\cong
E_n  \circ F_{n+1}\cong
F_n \circ S_n,
\end{equation}
(where $\cong$ means ordered isomorphism
between the edge sets preserving the
source and range maps).
Then by the assumption ~\eqref{prop_conj_it1}
 one can
eliminate $F_n$ from
\eqref{equ3_prop_conj} to get 
$ G_n  \circ F_{n+1}\cong
 S_n$. In fact, for every $w\in W_{n+1}$,
 by \eqref{equ3_prop_conj} we have
 \[
 \tau_n\big(\sigma_n(\tau_{n+1}(w))\big)=
 \tau_n\big(\sigma_{n+1}^{B_2}(w)\big)
 \]
and the latter (with the independence assumption)  implies that
$\sigma_n(\tau_{n+1}(w))=\sigma_{n+1}^{B_2}(w)$.
This gives \eqref{prop_conj_it3}
and finishes the proof of the claim.

Now consider the following ordered Bratteli
diagram $B$:
\[
\xymatrix{V_{0}\ar[r]^-{\resizebox{1.15em}{.65em}{$E_{1}$}}
 &V_{1}\ar[r]^-{\resizebox{1.15em}{.65em}{$F_{1}$}} 
 &W_{1}\ar[r]^-{\resizebox{1.15em}{.65em}{$G_{1}$}}&
 V_{2}\ar[r]^-{\resizebox{1.15em}{.65em}{$F_{2}$}}
 &W_{2}\ar[r]^-{\resizebox{1.15em}{.65em}{$G_{2}$}} 
 &V_{3}\ar[r]^-{\resizebox{1.15em}{.65em}{$F_{3}$}}
 &\cdots.
 }
\]
By \eqref{prop_conj_it3},
the telescoping of $B$ along the odd
levels (starting with $V_0$ as the 
zeroth level) is equivalent to
$B_1$, and by \eqref{prop_conj_it4},
the telescoping  along the even
levels  is equivalent to
$B_2$
(note that $E_1 \circ F_1\cong
F_0\circ S_1\cong S_1$
by the ordered commutativity of $f$).
Therefore,  $B_1$ is equivalent to $B_2$.
\end{proof}

Here is an example.

\begin{figure}
\begin{center}
\begin{tikzpicture}[scale=1.5]

\filldraw (12,10) circle [radius=0.1];
\filldraw (10,7) circle [radius=0.1];
\filldraw (12,7) circle [radius=0.1];
\filldraw (14,7) circle [radius=0.1];
\node at (10.35,7) {$u_1$};
\node at (12.3,7) {$v_1$};
\node at (14.3,7) {$w_1$};
\filldraw (10,4) circle [radius=0.1];
\filldraw (12,4) circle [radius=0.1];
\filldraw (14,4) circle [radius=0.1];
\node at (10.35,4) {$u_2$};
\node at (12.3,4) {$v_2$};
\node at (14.3,4) {$w_2$};
\filldraw (10,1) circle [radius=0.1];
\filldraw (12,1) circle [radius=0.1];
\filldraw (14,1) circle [radius=0.1];
\node at (10.35,1) {$u_3$};
\node at (12.3,1) {$v_3$};
\node at (14.3,1) {$w_3$};

\draw (11.85,9.9)--(9.95,7.12);
\draw (11.9,9.68)--(10.13,7.05);

\draw (11.9,9.68)--(11.9,7.12);
\draw (12.1,9.68)--(12.1,7.12);

\draw (12.1,9.68)--(13.87,7.05);
\draw (12.15,9.9)--(14.05,7.12);

\node at (10.85,8.6) {\tiny{1}};
\node at (11.15,8.4) {\tiny{5}};
\node at (10.96,8.51) {.};
\node at (11,8.48) {.};
\node at (11.04,8.45) {.};

\node at (11.83,8.3) {\tiny{1}};
\node at (12.17,8.3) {\tiny{9}};
\node at (11.95,8.3) {.};
\node at (12,8.3) {.};
\node at (12.05,8.3) {.};

\node at (13,8.2) {\tiny{1}};
\node at (13.28,8.43) {\tiny{13}};
\node at (13.08,8.28) {.};
\node at (13.12,8.32) {.};
\node at (13.16,8.36) {.};

\draw (10,6.87)--(10,4.14);
\draw (12,6.87)--(12,4.14);
\draw (13.98,6.87)--(13.98,4.14);
\draw (14.06,6.87)--(14.06,4.14);
\draw (11.97,6.85)--(10.08,4.1);
\draw (13.95,6.85)--(12.08,4.1);
\draw (12.02,6.85)--(13.96,4.14);
\draw (10.1,6.9)--(11.91,4.1);
\draw (10.15,6.95)--(13.94,4.14);

\draw (10,3.87)--(10,1.14);
\draw (12,3.87)--(12,1.14);
\draw (13.98,3.87)--(13.98,1.14);
\draw (14.06,3.87)--(14.06,1.14);
\draw (11.97,3.85)--(10.08,1.1);
\draw (13.95,3.85)--(12.08,1.1);
\draw (12.02,3.85)--(13.96,1.14);
\draw (10.1,3.9)--(11.91,1.1);
\draw (10.15,3.95)--(13.94,1.14);

\node at (9.92,4.35) {\tiny{1}};
\node at (10.15,4.35) {\tiny{2}};
\node at (11.78,4.4) {\tiny{1}};
\node at (11.93,4.37) {\tiny{2}};
\node at (12.15,4.35) {\tiny{3}};
\node at (13.53,4.52) {\tiny{1}};
\node at (13.77,4.52) {\tiny{2}};
\node at (13.91,4.52) {\tiny{3}};
\node at (14.15,4.52) {\tiny{4}};

\node at (9.92,1.35) {\tiny{1}};
\node at (10.15,1.35) {\tiny{2}};
\node at (11.79,1.4) {\tiny{1}};
\node at (11.93,1.37) {\tiny{2}};
\node at (12.15,1.35) {\tiny{3}};
\node at (13.53,1.52) {\tiny{1}};
\node at (13.77,1.52) {\tiny{2}};
\node at (13.91,1.52) {\tiny{3}};
\node at (14.15,1.52) {\tiny{4}};

\filldraw (6,10) circle [radius=0.1];
 \filldraw (5,7) circle [radius=0.1];
\filldraw (7,7) circle [radius=0.1];
 \node at (4.75,7) {$x_1$};
  \node at (6.75,7) {$y_1$};
 \filldraw (5,4) circle [radius=0.1];
\filldraw (7,4) circle [radius=0.1];
 \node at (4.75,4) {$x_2$};
  \node at (6.75,4) {$y_2$};
 \filldraw (5,1) circle [radius=0.1];
\filldraw (7,1) circle [radius=0.1];
 \node at (4.75,1) {$x_3$};
  \node at (6.75,1) {$y_3$};

\draw (5.91,9.92)--(4.96,7.12);
\draw (5.98,9.88)--(5.05,7.13);

\draw (6.07,9.91)--(6.97,7.12);

\draw (4.95,6.87)--(4.95,4.14);
\draw (4.95,3.87)--(4.95,1.14);
\draw (5.03,6.87)--(5.03,4.14);
\draw (5.03,3.87)--(5.03,1.14);

\draw (6.95,6.87)--(6.95,4.14);
\draw (6.95,3.87)--(6.95,1.14);
\draw (7.03,6.87)--(7.03,4.14);
\draw (7.03,3.87)--(7.03,1.14);

\draw (6.9,6.9)--(5.05,4.15);
\draw (6.9,3.9)--(5.05,1.15);
\draw (5.1,6.9)--(6.9,4.15);
\draw (5.1,3.9)--(6.9,1.15);


\node at (5.06,7.7) {\tiny{1}};
\node at (5.32,7.7) {\tiny{2}};

\node at (7.11,4.39) {\tiny{3}};
\node at (6.89,4.39) {\tiny{2}};
\node at (6.66,4.39) {\tiny{1}};
\node at (5.12,4.7) {\tiny{2}};
\node at (4.89,4.7) {\tiny{1}};
\node at (5.33,4.7) {\tiny{3}};
\node at (7.11,1.39) {\tiny{3}};
\node at (6.89,1.39) {\tiny{2}};
\node at (6.66,1.39) {\tiny{1}};
\node at (5.12,1.7) {\tiny{2}};
\node at (4.89,1.7) {\tiny{1}};
\node at (5.33,1.7) {\tiny{3}};

\draw[->, thick] (6.1,10.1) [out=10, in=170] to (11.9,10.1);

\draw[->, thick] (5.15,7) [out=10, in=170] to (9.85,7);
\draw[->, thick] (5.15,7.05) [out=10, in=170] to (9.85,7.1);
\draw[->, thick] (7.15,6.95) [out=-10, in=190] to (9.85,6.88);

\draw[->, thick] (5.15,7.16) [out=20, in=156] to (11.85,7.24);
\draw[->, thick] (5.15,7.12) [out=20, in=158] to (11.85,7.12);
\draw[->, thick] (5.15,7.08) [out=20, in=160] to (11.85,7);
\draw[->, thick] (7.15,6.9) [out=-15, in=200] to (11.85,6.94);
\draw[->, thick] (7.15,6.86) [out=-15, in=205] to (11.85,6.85);
\draw[->, thick] (7.15,6.83) [out=-15, in=210] to (11.85,6.76);

\draw[->, thick] (5.15,7.32) [out=30, in=154] to (13.66,7.3);
\draw[->, thick] (5.15,7.28) [out=30, in=156] to (13.66,7.18);
\draw[->, thick] (5.15,7.24) [out=30, in=158] to (13.66,7.06);
\draw[->, thick] (5.15,7.2) [out=30, in=160] to (13.66,6.94);
\draw[->, thick] (7.15,6.8) [out=-30, in=200] to (13.66,6.9);
\draw[->, thick] (7.15,6.77) [out=-30, in=204] to (13.66,6.8);
\draw[->, thick] (7.15,6.74) [out=-30, in=208] to (13.66,6.7);
\draw[->, thick] (7.15,6.71) [out=-30, in=212] to (13.66,6.6);
\draw[->, thick] (7.15,6.68) [out=-30, in=216] to (13.66,6.5);

\draw[->, thick] (5.15,4) [out=10, in=170] to (9.85,4);
\draw[->, thick] (5.15,4.05) [out=10, in=170] to (9.85,4.1);
\draw[->, thick] (7.15,3.95) [out=-10, in=190] to (9.85,3.88);

\draw[->, thick] (5.15,4.16) [out=20, in=156] to (11.78,4.24);
\draw[->, thick] (5.15,4.12) [out=20, in=158] to (11.85,4.12);
\draw[->, thick] (5.15,4.08) [out=20, in=160] to (11.85,4);
\draw[->, thick] (7.15,3.9) [out=-15, in=200] to (11.85,3.94);
\draw[->, thick] (7.15,3.86) [out=-15, in=205] to (11.85,3.85);
\draw[->, thick] (7.15,3.83) [out=-15, in=210] to (11.85,3.76);

\draw[->, thick] (5.15,4.32) [out=30, in=154] to (13.66,4.3);
\draw[->, thick] (5.15,4.28) [out=30, in=156] to (13.66,4.18);
\draw[->, thick] (5.15,4.24) [out=30, in=158] to (13.66,4.06);
\draw[->, thick] (5.15,4.2) [out=30, in=160] to (13.66,3.94);
\draw[->, thick] (7.15,3.8) [out=-30, in=200] to (13.66,3.9);
\draw[->, thick] (7.15,3.77) [out=-30, in=204] to (13.66,3.8);
\draw[->, thick] (7.15,3.74) [out=-30, in=208] to (13.66,3.7);
\draw[->, thick] (7.15,3.71) [out=-30, in=212] to (13.66,3.6);
\draw[->, thick] (7.15,3.68) [out=-30, in=216] to (13.66,3.5);

\draw[->, thick] (5.15,1) [out=10, in=170] to (9.85,1);
\draw[->, thick] (5.15,1.05) [out=10, in=170] to (9.85,1.1);
\draw[->, thick] (7.15,0.95) [out=-10, in=190] to (9.85,0.88);

\draw[->, thick] (5.15,1.16) [out=20, in=156] to (11.78,1.24);
\draw[->, thick] (5.15,1.12) [out=20, in=158] to (11.85,1.12);
\draw[->, thick] (5.15,1.08) [out=20, in=160] to (11.85,1);
\draw[->, thick] (7.15,0.9) [out=-15, in=200] to (11.85,0.94);
\draw[->, thick] (7.15,0.86) [out=-15, in=205] to (11.85,0.85);
\draw[->, thick] (7.15,0.83) [out=-15, in=210] to (11.85,0.76);

\draw[->, thick] (5.15,1.32) [out=30, in=154] to (13.66,1.3);
\draw[->, thick] (5.15,1.28) [out=30, in=156] to (13.66,1.18);
\draw[->, thick] (5.15,1.24) [out=30, in=158] to (13.66,1.06);
\draw[->, thick] (5.15,1.2) [out=30, in=160] to (13.66,0.94);
\draw[->, thick] (7.15,0.8) [out=-30, in=200] to (13.66,0.9);
\draw[->, thick] (7.15,0.77) [out=-30, in=204] to (13.66,0.8);
\draw[->, thick] (7.15,0.74) [out=-30, in=208] to (13.66,0.7);
\draw[->, thick] (7.15,0.71) [out=-30, in=212] to (13.66,0.6);
\draw[->, thick] (7.15,0.68) [out=-30, in=216] to (13.66,0.5);

\node at (9.6,7.22) {\tiny{1}};
\node at (9.6,6.96) {\tiny{2}};
\node at (9.6,6.75) {\tiny{3}};

\node at (11.5,7.47) {\tiny{1}};
\node at (11.5,7.27) {\tiny{2}};
\node at (11.5,7.1) {\tiny{4}};
\node at (11.5,6.9) {\tiny{3}};
\node at (11.5,6.72) {\tiny{5}};
\node at (11.5,6.5) {\tiny{6}};

\node at (13.28,7.57) {\tiny{1}};
\node at (13.28,7.38) {\tiny{2}};
\node at (13.28,7.25) {\tiny{4}};
\node at (13.28,7.1) {\tiny{7}};
\node at (13.28,6.85) {\tiny{3}};
\node at (13.38,6.7) {\tiny{5}};
\node at (13.28,6.54) {\tiny{6}};
\node at (13.38,6.43) {\tiny{8}};
\node at (13.38,6.21) {\tiny{9}};

\node at (9.6,4.22) {\tiny{1}};
\node at (9.6,3.96) {\tiny{2}};
\node at (9.6,3.75) {\tiny{3}};

\node at (11.5,4.45) {\tiny{1}};
\node at (11.5,4.27) {\tiny{2}};
\node at (11.5,4.1) {\tiny{4}};
\node at (11.5,3.9) {\tiny{3}};
\node at (11.5,3.72) {\tiny{5}};
\node at (11.5,3.5) {\tiny{6}};

\node at (13.2,4.57) {\tiny{1}};
\node at (13.28,4.38) {\tiny{2}};
\node at (13.28,4.25) {\tiny{4}};
\node at (13.28,4.1) {\tiny{7}};
\node at (13.28,3.85) {\tiny{3}};
\node at (13.38,3.7) {\tiny{5}};
\node at (13.28,3.54) {\tiny{6}};
\node at (13.38,3.43) {\tiny{8}};
\node at (13.38,3.21) {\tiny{9}};

\node at (9.6,1.22) {\tiny{1}};
\node at (9.6,0.96) {\tiny{2}};
\node at (9.6,0.75) {\tiny{3}};

\node at (11.5,1.45) {\tiny{1}};
\node at (11.5,1.27) {\tiny{2}};
\node at (11.5,1.1) {\tiny{4}};
\node at (11.5,0.9) {\tiny{3}};
\node at (11.5,0.72) {\tiny{5}};
\node at (11.5,0.5) {\tiny{6}};

\node at (13.2,1.57) {\tiny{1}};
\node at (13.28,1.38) {\tiny{2}};
\node at (13.28,1.25) {\tiny{4}};
\node at (13.28,1.1) {\tiny{7}};
\node at (13.28,0.85) {\tiny{3}};
\node at (13.38,0.7) {\tiny{5}};
\node at (13.28,0.54) {\tiny{6}};
\node at (13.38,0.43) {\tiny{8}};
\node at (13.38,0.21) {\tiny{9}};

\node at (6,10.7) {$C $};
\node at (12,10.7) {$B' $};
\node[very thick] at (6,0.2) {\vdots};
\node[very thick] at (12,-0.5) {\vdots};
\end{tikzpicture}
\end{center}
\caption{The ordered premorphism $g\colon C\to B'$.}\label{fig_Chacon2}
\end{figure}
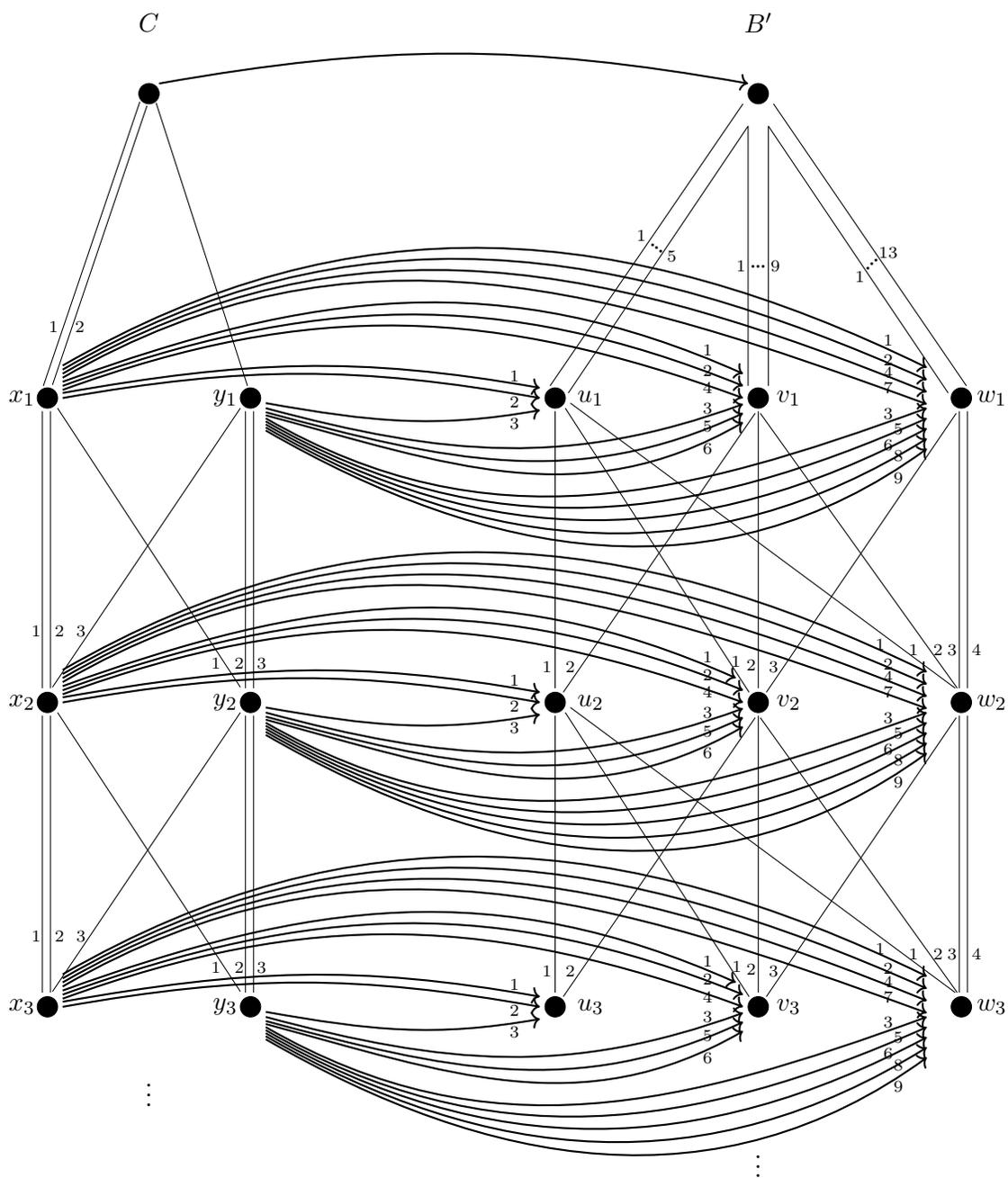

\begin{example}\label{exa_conj}
Consider the ordered premorphism $g\colon C\to B'$ of \cite[Example~4.15]{AEG}.
Its diagram is drawn in Figure~4 of \cite{AEG} and we draw it
here in Figure~\ref{fig_Chacon2} for the convenience of the reader. 
Let us recall that $C$ is the Bratteli diagram of the 
Chacon system  and $B$ is a simple properly ordered Bratteli diagram. 
By Proposition~\ref{AEG_prop}, $g$ induces a factor map
$\alpha\colon X_{B'}\to X_C$.
We use Proposition~\ref{prop_conj} to show that $X_C$ is conjugate
to $X_{B'}$. (In \cite{AEG}, another argument is given to show that these 
systems are conjugate by constructing the inverse of $g$.)
We label the $i$-th level of $C$ by $x_i$ and $y_i$, and that of
$B'$ by $u_i$, $v_i$, and $w_i$, for $i\geq 1$. Let $(\tau_i)_{i\geq 0}$
be the sequence of morphisms induced by $g$ according to Definition~\ref{new}.
For every $i\geq 1$ we have
\[
\tau_{i}(u_i)=x_{i}x_{i}y_{i},\ \ \ \tau_i(v_i)=x_{i}x_{i}y_{i}x_{i}y_{i}y_{i}, 
\ \text{and} \ \ \tau_i(w_i)=x_{i}x_{i}y_{i}x_{i}y_{i}y_{i}x_{i}y_{i}y_{i}.
\]
Put $D_i=\{\tau_{i}(u_i),\tau_{i}(v_i),\tau_{i}(w_i)    \}$.
It is easy to check that $D_i$ is independent.
Moreover, for every $i\geq 1$ we have
\[
\sigma^{C}_{i+1}(x_{i+1})=x_{i}x_{i}y_{i}\ \ \
\text{and} \ \ \
\sigma^{C}_{i+1}(y_{i+1})=x_{i}y_{i}y_{i}.
\]
We see that $D_i$ does not generate both of the words above. So, we go one level down
on  diagram $C$ and compute
\[
\sigma^{C}_{[i,i+2]}(x_{i+2})=x_{i}x_{i}y_{i}x_{i}x_{i}y_{i}x_{i}y_{i}y_{i}\ \ 
\text{and} \ \
\sigma^{C}_{[i,i+2]}(y_{i+2})=x_{i}x_{i}y_{i}x_{i}y_{i}y_{i}x_{i}y_{i}y_{i}.
\]
These two words are generated by $D_i$ as
$\sigma^{C}_{[i,i+2]}(x_{i+2})=\tau_{i}(u_i)\tau_{i}(v_i)$ and 
$\sigma^{C}_{[i,i+2]}(y_{i+2})=\tau_{i}(w_i)$. Also, $D_i$
is a minimal subset of $\{x_i,y_i.z_i\}^{*}$ having this property.
Now, Proposition~\ref{prop_conj} implies that
$C$ is equivalent to $B'$, and so $X_C$ is conjugate
to $X_{B'}$.
\end{example}

%


\end{document}